\documentclass[reqno,11pt]{amsart}
\usepackage{amsthm,amsfonts,amssymb,euscript,
mathrsfs,graphics,color,amsmath,latexsym,marginnote}
 
\theoremstyle{plain}

\usepackage{hyperref}
\usepackage{times}

\numberwithin{equation}{section}

\newtheorem{theorem}{Theorem}[section]
\newtheorem{proposition}[theorem]{Proposition}
\newtheorem{lemma}[theorem]{Lemma}
\newtheorem{corollary}[theorem]{Corollary}

\newtheorem{remark}[theorem]{Remark}
\newtheorem{remarks}[theorem]{Remark}

\newtheorem{definition}[theorem]{Definition}
\newcommand{\be}{\begin{equation}}
\newcommand{\ee}{\end{equation}}
\newcommand{\om}{\omega}
\newcommand{\eps}{\varepsilon}
\newcommand{\e}{\varepsilon}

\newcommand{\ov}{\overline}

\newcommand{\R}{\mathbb R}
\newcommand{\C}{\mathbb C}

\newcommand{\N}{\mathbb N}

\renewcommand{\b }{\beta }
\newcommand{\s }{\sigma }
\newcommand{\ii }{{\rm i} }

\newcommand{\g }{\gamma}

\newcommand{\vphi}{\varphi }

\renewcommand{\o }{\omega }

\newcommand{\bral}{[ \! [} 
\newcommand{\brar}{] \! ]}

\newcommand{\pa}{\partial}

\def\bar{\overline}

\newcommand{\Rc}{\mathcal R}

\renewcommand{\S}{\mathbb{S}}

\def\ba{\begin{aligned}}
\def\ea{\end{aligned}}
\def\beginm{\begin{multline}}
\def\endm{\end{multline}}

\makeatletter

\def\l@subsection{\@tocline{2}{0pt}{2.5pc}{5pc}{}}
\def\l@subsubsection{\@tocline{3}{0pt}{4.5pc}{5pc}{}}
\renewcommand\tocchapter[3]{%
  \indentlabel{\@ifnotempty{#2}{\ignorespaces#2.\quad}}#3%
}
\newcommand\@dotsep{4.5}
\def\@tocline#1#2#3#4#5#6#7{\relax
  \ifnum #1>\c@tocdepth 
  \else
    \par \addpenalty\@secpenalty\addvspace{#2}%
    \begingroup \hyphenpenalty\@M
    \@ifempty{#4}{%
      \@tempdima\csname r@tocindent\number#1\endcsname\relax
    }{%
      \@tempdima#4\relax
    }%
    \parindent\z@ \leftskip#3\relax \advance\leftskip\@tempdima\relax
    \rightskip\@pnumwidth plus1em \parfillskip-\@pnumwidth
    #5\leavevmode\hskip-\@tempdima{#6}\nobreak
    \leaders\hbox{$\m@th\mkern \@dotsep mu\hbox{.}\mkern \@dotsep mu$}\hfill
    \nobreak
    \hbox to\@pnumwidth{\@tocpagenum{#7}}\par
    \nobreak
    \endgroup
  \fi}
\makeatother
\AtBeginDocument{%
\makeatletter
\expandafter\renewcommand\csname r@tocindent0\endcsname{0pt}
\makeatother
}
\def\l@subsection{\@tocline{2}{0pt}{2.5pc}{5pc}{}}

\begin{document}

\title[Reducibility of LS on the Sphere]{Reducibility of  Schr\"odinger equation 
on the Sphere.}

\date{}

\author{Roberto Feola}
\address{\scriptsize{Laboratoire de Math\'ematiques Jean Leray, Universit\'e 
de Nantes, 
UMR CNRS 6629
\\
2, rue de la Houssini\`ere 
\\
44322 Nantes Cedex 03, France}
}
\email{roberto.feola@univ-nantes.fr}

\author{Beno\^it Gr\'ebert}
\address{\scriptsize{Laboratoire de Math\'ematiques Jean Leray, Universit\'e de Nantes, UMR CNRS 6629\\
2, rue de la Houssini\`ere \\
44322 Nantes Cedex 03, France}}
\email{benoit.grebert@univ-nantes.fr}

\thanks{During the preparation of this work 
the two authors benefited from the support 
of the Centre Henri Lebesgue ANR-11-LABX- 0020-01 and  of 
ANR -15-CE40-0001-02  ``BEKAM''  of the Agence 
Nationale de la Recherche.  
R. F. was also supported by ERC starting grant FAFArE of the European Commission and
 B.G. 
 by  ANR-16-CE40-0013 ``ISDEEC'' of the Agence Nationale de la Recherche.}

\begin{abstract} 
In this article we prove a reducibility result for the linear Schr\"odinger equation on the sphere $\S^n$ with  quasi-periodic in time perturbation. Our result includes the case of unbounded perturbation that we assume to be of order strictly less than $1/2$ and satisfying some parity condition. As far as we know, this is one of the few reducibility results for an equation in more than one dimension with unbounded perturbations. We notice that our result does not requires the use of the pseudo-differential calculus.
\end{abstract}  
  
\maketitle

\tableofcontents

\section{Introduction}
In this article we are interested in the problem of reducibility for the linear Schr\"odinger equation on the sphere with  quasi-periodic in time perturbation.
In the introduction,  to make our statement more readable, we state our results in the physical space $\S^2\subset \R^3$ and with an explicit  linear perturbation. A more general statement, including the higher dimension case $\S^n$ for $n\geq3$ (the case $n=1$ is much more simpler and already known, see \cite{Kuk93}), is detailed in section \ref{AppliNLS}. So we consider the following linear Schr\"odinger equation on the $\mathbb{S}^{2}$
\[\label{nls}\tag{LS}
\ii \pa_{t}u=\Delta u+\eps\big( W(\o t,x)(-\ii\partial_\phi)^\alpha+V(\omega t,x)\big)u\,,\qquad u=u(t,x)\,,
\quad t\in \mathbb{R}\,,\quad x\in \mathbb{S}^{2}\,,
\]
where $\Delta$ denotes the Laplace-Beltrami operator on $\mathbb{S}^{2}$,  $\ii\partial_\phi=\ii(x_1\partial_{x_2}-x_2\partial_{x_1})$ is the 
$x_3$ component of the orbital angular momentum 
(and the generator of rotations about the $x_{3}$ axis) 
and $0\leq\alpha<1/2$. The operator $(-\ii\partial_\phi)^\alpha$ 
is precisely defined in \eqref{partialalfa}.
The parameter $\eps>0$ is  small, the frequency vector $\omega$ 
 belongs to $\mathcal{O}_0:=[1/2,3/2]^{d}\subset \mathbb{R}^{d}$, $d\geq1$.
The functions $W,\,V$ in \eqref{nls}
are  real valued multiplicative potentials depending quasi periodically on 
time, i.e. $V$ is a function in $C^{0}(\mathbb{T}^{d}\times \mathbb{S}^{2};\mathbb{R})$, $\mathbb{T}:=\mathbb{R}/2\pi \mathbb{Z}$.
We assume that $W,\,V$ are real analytic functions
with respect to the angle variable $\vphi\in\mathbb{T}^{d}$ with values in
the Sobolev space
$H^{s}(\mathbb{S}^{2};\mathbb{R})$ with $s>d/2+1$.
In particular the map 
$\mathbb{T}^{d}\ni\vphi\mapsto V(\vphi,\cdot)\in 
H^{s}(\mathbb{S}^{2};\mathbb{R})$ analytically extends to
\begin{equation}\label{toroidal}
\mathbb{T}^{d}_{\s}:=\big\{ (a+\ii b)\in \mathbb{C}^{d}\, : \, 
a\in \mathbb{T}^{d}\,, |b|<\s\big\}\,,
\end{equation}
for some $\s>0$.\\
We stress out that \eqref{nls} doesn't describe the most general case that we can consider, in particular $\partial_\phi$ could be replaced by some
unbounded operator. 

The purpose of reducibility is to construct a change of variables that transforms the non-autonomous equation \eqref{nls} into an autonomous equation.\\
Our main result is the following.
\begin{theorem}\label{thm:main}  Let $0<\delta<1$ and $0\leq \alpha<1/2$. 
Assume that 
$\vphi\mapsto V(\vphi,\cdot)\in 
H^{s+s_0}(\mathbb{S}^{2};\mathbb{R})$ and 
$\vphi\mapsto W(\vphi,\cdot)\in 
H^{s+s_0}(\mathbb{S}^{2};\mathbb{R})$ 
analytically extend to 
$\mathbb{T}^{d}_{\s}$ for some 
$\s>0$, $s>0$ and 
for some  $s_0=s_0(d,\alpha)\geq(d+2)/2$ large enough.
Assume furthermore that the potentials $V$ and 
$W$ are \emph{odd} in the variable
$x\in \mathbb{S}^{2}$. 
 There exists  $\eps_0>0$ such that, for any $0<\eps\leq \eps_0$ there is
 a set $\mathcal{O}_\eps\subset\mathcal{O}_0\subset\mathbb{R}^{d}$
 with 
 \begin{equation}\label{measure}
 {\rm meas}(\mathcal{O}_0\setminus\mathcal{O}_\eps)\leq \eps^{\delta}
 \end{equation}
 such that the following holds:

 \noindent
 for any $\omega\in \mathcal{O}_\e$
 there exists 
 a family  linear isomorphisms 
 $\Psi(\vphi)\in \mathcal{L}(H^{s}(\mathbb{S}^{2};\mathbb{C}))$,
 analytically 
 depending on 
 $\vphi\in \mathbb{T}^{d}_{\s/2}$ 
 and a Hermitian operator 
 $Z\in \mathcal{L}(H^{s}(\mathbb{S}^{2};\mathbb{C});
 H^{s+1-\alpha}(\mathbb{S}^{2};\mathbb{C}))$ commuting with the Laplacian and
satisfying 

 $\bullet$
 $\Psi(\vphi) $ is unitary on $L^{2}(\mathbb{S}^2;\mathbb{C})$;
 
 $\bullet$
  for any $0\leq s'\leq s$
 \begin{equation}\label{stima}
 \|\Psi(\vphi)-{\rm Id}\|_{\mathcal{L}(H^{s'}(\mathbb{S}^{2};\mathbb{C}))}+
  \|\Psi(\vphi)^{-1}-{\rm Id}\|_{\mathcal{L}(H^{s'}(\mathbb{S}^{2};\mathbb{C}))}
  \leq \e^{1-\delta}\,,
 \end{equation}
 
 $\bullet$
 the function $t\mapsto u(t,\cdot)\in H^{s'}(\mathbb{S}^{2};\mathbb{C})$
 solves \eqref{nls} if and only if the map
 $t\mapsto v(t,\cdot):=\Psi(\omega t)u(t,\cdot)$ solves the autonomous 
 equation
 \begin{equation}\label{NLSred}
 \ii \pa_t v=\Delta v+\eps Z(v)\,.
 \end{equation}
\end{theorem}

As a consequence of our reducibility result, we get 
the following corollary concerning the solutions of \eqref{nls}. 
  \begin{corollary}\label{coro1.3}
 Assume that $V$ and $W$ satisfy the same assumptions than in Theorem \ref{thm:main}. Let $1\leq s'\leq s$ and let  
 $u_{0}\in H^{s'}(\mathbb{S}^{2};\mathbb{C})$. 
 Then there exists $\eps_{0}>0$ such that for all 
 $0<\eps<\eps_{0}$ and for all $\o \in \mathcal{O}_{\eps}$,  
 there exists a unique solution $u \in {C}^{0}\big(\R\,;\,H^{s'}\big)$ 
 of \eqref{nls} such that $u(0)=u_{0}$. 
 Moreover, $u$ is almost-periodic in time and satisfies
  \begin{equation}\label{16}
  (1-\eps C)\|u_{0}\|_{H^{s'}}\leq \|u(t)\|_{H^{s'}}\leq  (1+\eps C)\|u_{0}\|_{H^{s'}}, \quad \forall \,t\in \R,
  \end{equation}
  for some $C=C(s',s,d)$.
  \end{corollary}

The study of the reducibility problem for Schr\"odinger equations 
 with quasi-periodic in time perturbation has been 
 very popular in recent years. 
 The first results adapting the KAM technics 
 were due to Kuksin \cite{Kuk87, Kuk93} 
 (see also \cite{Wayne,KuPo,Po2,BG, LY, Gtom}) 
 and concerned only one dimensional case. 
 More recently the technics 
were adapted to the higher dimensional case \cite{EK, EGK, GPatu, PP}. 
To consider unbounded perturbations, a new strategy has been developed in \cite{BBM14, KdVAut} using  the pseudo-differential calculus. 
 Without trying to be exhaustive 
we quote also \cite{FP1, BM1, BBHM, FGP} regarding KAM theory for 
quasi-linear PDEs in one space dimension.
This technics were successfully applied for reducibility problems in various case. 
For one dimensional linear equations with unbounded potential we quote
\cite{Bam17, Bam18, BM18, FGP1}.
In higher space dimensions
we refer to \cite{EK09, GP}
for bounded potential, and to
\cite{BGMR18, Mon18, FGMP, BLM18} for the unbounded cases.

In this paper we choose to present an intermediate result were pseudo-differential calculus is not required although the perturbation is unbounded. We believe that the simplicity of this paper justifies this choice.

\medskip

\noindent
{\bf Scheme of the proof.}
We now briefly describe the structure of the proof.
Some key points concern 
\begin{itemize}
\item[1)] the matrix representation of the multiplication operator 
 $u\mapsto bu$
by a function $b\in H^{s}(\mathbb{S}^{n};\mathbb{C})$;

\item[2)] the properties of the Laplace-Beltrami operator on $\mathbb{S}^{n}$;

\item[3)] a sufficiently accurate asymptotic expansion of the eigenvalues of the
linear  operator
in the right hand side of \eqref{nls}.
\end{itemize}
%

Regarding item $1)$,  the key property which is exploited is that
the product of two eigenfunctions is a finite linear combinations of them.
Hence the  rule of multiplications of the eigenfunctions 
implies  that the multiplication operator $u\mapsto bu$ 
can be represented, in the base of eigenfunctions,
as a block matrix with off-diagonal decay.
The block structure of this matrix is a consequence of the
multiplicity of the eigenvalues of $\Delta$ on $\mathbb{S}^{n}$.
For the analysis of these decay properties we refer to 
\cite{BP} and \cite{BCP} in which it is considered the more general case
of equations on Lie Group or on 
compact manifolds  which are homogenenous 
with respect to a compact Lie Group.
 In \cite{GP} the use of these decay-norms was not possible since in the case of the quantum harmonic operator we need to use specific dispersive properties of the eigenfunctions.

Concerning item $2)$, we strongly use the fact that
the eigenvalues
$\lambda_{k}$, $k\in \mathbb{N}$ of $\Delta$ (see \eqref{eigen})
satisfy a very strong ``separation property'' i.e.
\begin{equation}\label{sepProp}
|\lambda_{k}-\lambda_{k'}|\geq k+k'\,,\qquad \forall\, k,k'\in \mathbb{N}\,, k\neq k'\,.
\end{equation}
These property holds for the Laplace-Beltrami operator on $\mathbb{S}^{n}$
and more in general holds for compact manifolds  which are homogenenous 
with respect to a compact Lie Group of rank $1$. 
We remark that this property is not true for ``any'' homogeneous manifold. For 
instance, it is violated by the eigenvalues of $\Delta$ on the
torus $\mathbb{T}^{n}$, $n\geq2$,
which have the form $|j|^{2}$ with $j\in \mathbb{Z}^{d}$.
The separation property in \eqref{sepProp}
is deeply used in the preliminary regularization step
in section \ref{regu}.
In this step we also require an oddness hypothesis
on the multiplicative potential $W$, $V$.

To understand the use of 
item $3)$ we briefly discuss the difficulties related to 
reducibility in high space dimension.
We first recall that the Laplace 
operator $\Delta$ diagonalizes on the basis of
the \emph{spherical harmonics} of the sphere $\mathbb{S}^{n}$. We denote
by $E_{k}$ the eigenspace associated to the eigenvalues $\lambda_{k}$ 
(see \eqref{eigen}), $k\in \mathbb{N}$. 
It is also know that the dimension of $E_{k}$
grows to infinity as $k\to \infty$. 
We shall denote by $\Phi_{k,j}$, $j=1,\ldots,d_{k}:=\dim E_{k}$
an orthonormal basis of $E_{k}$. 
With this formalism, the matrix $A$, which represents
the operator $W(\o t,x)(-\ii\partial_\phi)^\alpha+V(\omega t,x)$ (see \eqref{nls})
in the basis $\Phi_{k,j}$, has the form
$A:=A(\omega t):=\Big(A_{[k]}^{[k']}\Big)_{k,k'\in\mathbb{N}}$
with blocks $A_{[k]}^{[k']}\in \mathcal{L}(E_{k'};E_{k})$.
The reducibility of \eqref{nls} rely on the reducibility of the operator $\Delta+\eps A$
which is divided into two steps.

The first one is to regularize the \eqref{nls} equation 
to a Schr\"odinger equation 
with a smoothing quasi-periodic in time perturbation.
This is the content of section \ref{regu}.
More precisely, using also the oddness assumption on the potentials,
we are able to show that
$i)$  the operator $\Delta+\eps A$
 can be conjugated to an operator of the form
\begin{equation}\label{asympexp2}
\Delta+\eps M\,,\qquad M : H^{s}(\mathbb{S}^{n};\mathbb{C})\to 
H^{s+1-2\alpha}(\mathbb{S}^{n};\mathbb{C})\,,
\end{equation}
and $ii)$
the 
eigenvalues of $\Delta+\eps M$ have the form
\begin{equation}\label{asympexp}
\Lambda_{k,j}\sim \lambda_{k}+O(\e k^{-(1-2\alpha)})\,, \qquad 
j=1,\ldots,d_{k}\,.
\end{equation}
We remark that, since $\alpha<1/2$, the matrix $M$ in \eqref{asympexp2} 
is a ``regularizing'' operator, and its eigenvalues in \eqref{asympexp} 
are very ``close'' to the unperturbed eigenvalues $\lambda_k$.

The second part of the proof 
consists in a quite standard KAM step following \cite{GP} or \cite{EK09}.
We note that in this second step we use the decay-norms 
introduced in \cite{BP} (see also \cite{BCP})  which provides a simpler algebraic framework.
A key point of a reducibility scheme is the resolution 
of the so called ``homological equation'', which relies on the invertibility of
an infinite dimensional matrix which is 
block diagonal with respect to the orthogonal splitting $L^{2}=\oplus_{k\in \mathbb{N}}E_{k}$ (see \eqref{orthosplitto}).
The fact that $\dim E_{k}\sim k^{n-1}$
makes hard the control of the inverse of such matrix, and could, in principle, creates
loss of regularity at each step of the iteration.
To overcome this problem we take advantages of the 
regularizing effect of the matrix $M$ to solve the homological equation
using a trick previously used in literature, see for instance \cite{GPatu, GP}.
We refer the reader to Lemma \ref{omoequation}
where the properties \eqref{asympexp2}, \eqref{asympexp} are used to prove 
suitable estimates on the solution of the homological equation
(see 
the
 bound \eqref{claimo}).
We remark that, in \cite{GP}, the regularizing effect of the perturbations
is proved by using special dispersive properties of the 
eigenfunctions which do not hold in our context.

It is also know that reducibility of a matrix $M$ (even in finite dimension)
requires some non-degeneracy conditions on differences of two eigenvalues, 
the so called
``second order Melnikov conditions''. 
More precisely
we shall prove that, for ``most'' parameters $\omega$, one has   lower bounds 
 of the form
\begin{equation}\label{secondMel}
|\omega\cdot l+\Lambda_{k,j}-\Lambda_{k',j'}|\geq \frac{\gamma}{|l|^{\tau}}\,,
 \quad l\in \mathbb{Z}^{d}\,,
 \;\;k,k'\in \mathbb{N}
\end{equation}
and $j=1,\ldots,d_{k}$, $j'=1,\ldots,d_{k'}$ (see \eqref{calO+} for more details).
In order to prove that the set 
of ``good'' parameters has large Lebesgue measure
it is fundamental to show that 
for any fixed $l\in \mathbb{Z}^{d}$, there are only \emph{finitely many}
indexes $k,k'$ such that the conditions \eqref{secondMel} are violated.
Since the asymptotic of the eigenvalues in \eqref{asympexp} is superlinear,
i.e. $\sim k^{\mathtt{d}}$ with $\mathtt{d}>1$,
it is quite easy 
to show that the \eqref{secondMel} are violated only if $k+k'\leq |l|$. 
The case $k=k'$ is more delicate and  the 
asymptotic \eqref{asympexp} play a fundamental role.
For more details we refer the reader to Lemma \ref{lem:misuro}.

\vspace{0.9em}
We note that the regularization of section \ref{regu} could be obtained by using a pseudo-differential calculus in the spirit of \cite{BBM14}. Actually in a subsequent paper we will extend our result using the regularization procedure developed in \cite{BGMR2}. We expect to generalized Theorem \ref{thm:main} to the case of a quasi-periodic in time perturbation of order less or equal than $1/2$.

\vspace{2em}
\noindent{\bf Acknowledgments.} 
The authors wish to thank M. Procesi for many useful discussions.


\section{Functional setting}

In this section we introduce the space of functions, sequences 
and linear operators we shall use along the paper.
We shall write $a\leq_{s} b$ to denote
$a\leq C b$ for some constant $C=C(s,d,n)$ depending only 
on $s,d,n$
(which are fixed parameters of the problem).

\subsection{Space of functions and sequences}
We denote by ${\mathcal{E}}:=\{\lambda_k\,, k\in \mathbb{N}\}$
with 
\begin{equation}\label{eigen}
\lambda_{k}:=k(k+n-1)\,, \quad k\in \mathbb{N}
\end{equation}
the spectrum of 
 $-\Delta$ where $\Delta$ is the Laplace-Beltrami operator on the
sphere $\mathbb{S}^{n}$ and let $E_{k}$ be the eigenspace
associated to $\lambda_{k}$. 
We have
\begin{equation}\label{dimension}
{\rm dim}E_{k}:=d_{k}\leq k^{n-1}\,.
\end{equation}
We denote by 
\begin{equation}\label{ortobase}
\Phi_{[k]}(x):=\{\Phi_{k,m}(x)\,, m=1,\ldots,d_k\}
\end{equation}
an orthonormal  
basis of $E_k$ 
so that any function $u\in L^{2}(\mathbb{S}^{n};\mathbb{C})$
can be written as
\begin{equation}\label{FouExp}
u(x)=\sum_{k\in \mathbb{N}}\sum_{m=1}^{d_k}z_{k,m}\Phi_{k,m}=\sum_{k\in \mathbb{N}}z_{[k]}\cdot\Phi_{[k]}(x)\,,\qquad z_{[k]}=(z_{k,1},\cdots,z_{k,d_k})\in \mathbb{C}^{d_{k}}\,,
\end{equation}
where $''\cdot ''$ denotes the usual scalar product in $\mathbb{R}^{d_{k}}$.
We denote 
by $\Pi_{E_k}$ the $L^{2}$-projector 
on the eigenspace $E_k$, i.e. 
\begin{equation}\label{proiettore}
(\Pi_{E_{k}}u)(x)=z_{[k]}\cdot \Phi_{[k]}(x)
\qquad \Rightarrow \qquad 
(-\Delta)\Pi_{E_k}u=\lambda_{k}\Pi_{E_k}u\,, \quad k\in\mathbb{N}\,.
\end{equation}

For $s\geq0$, we define the (Sobolev) scale of Hilbert sequence
spaces
\begin{equation}\label{spseq}
h^{s}:=\big\{ z=\{z_{[k]}\}_{k\in \mathbb{N}}\,, z_{[k]}\in 
\mathbb{C}^{d_{k}}\,:\,  \|z\|^{2}_{s}:=
\sum_{k\in \mathbb{N}} \langle k\rangle^{2s}\|z_{[k]}\|^{2}  <+\infty \big\}\,,
\end{equation}
where $\langle k\rangle:= \max\{1,|k|\}$ and $\|\cdot\|$ 
denotes
the $L^{2}(\mathbb{C}^{d_{k}})$-norm.
By a slight abuse of notation we define the operator $\Pi_{E_k}$ on sequences 
as
$\Pi_{E_k}z=z_{[k]}$ for any $  z\in h^{s}$
and $k\in \mathbb{N}$. \\
We note that 
$$H^s=H^s(\S^n,\C):=\{u(x)= \sum_{k\in \mathbb{N}}z_{[k]}\cdot\Phi_{[k]}(x)\mid z\in h^s\}$$
is the standard Sobolev space  and
$\|u\|_s:=\|z\|_s$ is equivalent to the standard Sobolev norm.
\begin{remark}
First of all notice that the weight $\langle k\rangle$ 
we use in the norm in 
\eqref{spseq} is related 
to the eigenvalues of the Laplace-Beltrami 
operator, indeed
\begin{equation}\label{equiweight}
c|k|\leq  \sqrt{\lambda_{k}}\leq C|k|
\end{equation}
for some suitable constants $0<c\leq C$. 
\end{remark}
In the paper we shall also deal with 
functions of the space-time  $u(\vphi,x)$ 
which can be expanded, using the standard Fourier theory, as
\begin{equation}\label{decompogo}
u(\vphi,x)=\sum_{\ell\in \mathbb{Z}^{d}, k\in \mathbb{N}}z_{[k]}(l)\cdot
 \Phi_{[k]}(x) e^{\ii l \cdot\vphi}\,, \qquad z_{[k]}(l)\in \mathbb{C}^{d_k}
\end{equation}
where $e^{\ii l\cdot\vphi} \Phi_{k,m}(x)$, 
$l\in \mathbb{Z}^{d}$, $k\in \mathbb{N}$, $m=1,\ldots,d_k$
is an orthogonal basis of 
$L^{2}(\mathbb{T}^{d}\times \mathbb{S}^{n};\mathbb{C})$.
For $p_1,s_1\geq0$ 
we define the space $H^{p_1}(\mathbb{T}_{\s}^{d}; 
H^{s_1}(\mathbb{S}^{n};\mathbb{C}))$
as the space of functions 
\[ 
\mathbb{T}_{\s}^{d}\ni\vphi\mapsto H^{s_1}(\mathbb{S};\mathbb{C})\,\;\;
{\rm analytic\;\; for }\; |{\rm Im}(\vphi)|<\s\,,\qquad
p_1-{\rm Sobolev}\;\; {\rm for } \;\; |{\rm Im}(\vphi)|=\s\,.
\]
We shall work with functions $u(\vphi,x)$ in the space 
$\mathcal{A}_{s,\s}$, $s\geq0$, $\s>0$, 
\begin{equation}\label{timefunc}
\mathcal{A}_{s,\s}:=\bigcap_{p_1+s_1=s}
H^{p_1}(\mathbb{T}_{\s}^{d}; 
H^{s_1}(\mathbb{S}^{n};\mathbb{C}))
\end{equation}
which we identify (using \eqref{decompogo}) with 
the  
space of sequence
\begin{equation}\label{timeseq}
\ell_{s,\s}:=\big\{
z=\{z_{[k]}(l)\}_{l\in \mathbb{Z}^{d},k\in\mathbb{N}}\,,\, z_{[k]}\in 
\mathbb{C}^{d_k}:\,
\|z\|_{s,\s}^{2}:=\sum_{l\in \mathbb{Z}^{d},k\in \mathbb{N}}
\langle l, k\rangle^{2s}e^{2|l|\s}\|z_{[k]}(l)\|^{2}< +\infty
\big\}
\end{equation}
and we endow the space $\mathcal{A}_{s,\s}$ with the norm
$\|u\|_{\mathcal{A}_{s,\s}}:=\|z\|_{s,\s}$.

\vspace{0.9em}
\noindent
{\bf Lipschitz norm.} Consider a compact subset $\mathcal{O}$ of
$\mathbb{R}^{d}$, $d\geq1$. 
For functions $f : \mathcal{O}\to E$, with $(E,\|\cdot\|_{E})$
some Banach space, we define 
the sup norm and the lipschitz semi-norm
as
\begin{equation}\label{suplip}
\begin{aligned}
\|f\|_{E}^{sup}:=\|f\|_{E}^{sup,\mathcal{O}}:=\sup_{\omega\in \mathcal{O}}
\|f(\omega)\|_{E}\,,\qquad
\|f\|_{E}^{lip}:=\|f\|_{E}^{lip,\mathcal{O}}:=\sup_{\substack{\omega_1,\omega_2\in \mathcal{O}\\ \omega_1\neq\omega_2}}\frac{\|f(\omega_1)-f(\omega_2)\|_{E}}{|\omega_1-\omega_2|}\,.
\end{aligned}
\end{equation}
For any $\gamma>0$ we introduce the weighted Lipschitz norms
\begin{equation}\label{Lip-norm}
\|f\|_{E}^{\gamma,\mathcal{O}}:=\|f\|_{E}^{sup,\mathcal{O}}
+\gamma \|f\|_{E}^{lip,\mathcal{O}}\,.
\end{equation}
In order to simplify the notation,
if $E=\ell_{s,\s}$ in \eqref{timeseq}, we shall write
\begin{equation}\label{Lip-norm2}
\|f\|_{\ell_{s,\s}}^{\gamma,\mathcal{O}}=:\|f\|_{s,\s}^{\gamma,\mathcal{O}}
=\|f\|_{s,\s}^{sup,\mathcal{O}}
+\gamma \|f\|_{s,\s}^{lip,\mathcal{O}}\,.
\end{equation}
%
%
We finally define the space of sequences 
\begin{equation}\label{spaseqLip}
\ell_{s,\s}^{\gamma,\mathcal{O}}:=\big\{
\mathcal{O}\ni\omega\mapsto z(\omega)\in\ell_{s,\s}\,:\, 
\|z\|_{s,\s}^{\gamma,\mathcal{O}}<+\infty
\big\}\,.
\end{equation}
We have the following Lemma.

\begin{lemma}\label{AlgebNorma}
For $s> (d+n)/2$, for any $z,v\in \ell_{s,\s}$ 
there is $C(s)>0$ such that

\noindent
(1) {Sobolev embedding:} $\|z\|_{L^{\infty}}\leq C(s)\|z\|_{s,\s}\,$;

\noindent
(2) {algebra:}  $\|zv\|_{s,\s}\leq C(s)\|z\|_{s,\s}\|v\|_{s,\s}$.


\noindent
(3) Setting, for $N>0$,  $\Pi_{N}z=\{z_a(l)\}_{|l|\leq N,a\in \mathcal{E}}$,
one has
\begin{equation}\label{smoothstim}
\begin{aligned}
\|({\rm Id}-\Pi_{N})z\|_{s,\s'}\leq \frac{C(s)e^{-(\s-\s')N}}{(\s-\s')^{d}}\|z\|_{s,\s}\,.
\end{aligned}
\end{equation}
Similar bounds holds also replacing $\|\cdot\|_{s,\s}$
with the norm $\|\cdot\|_{s,\s}^{\gamma,\mathcal{O}}$.
\end{lemma}

\begin{proof}
Items $(1)$ and $(2)$ are classical estimates for Sobolev spaces, see for instance
Lemma $2.13$ in \cite{BP}.
\noindent
Item $(3)$ follows by the definition of the norm.
\end{proof}

\subsection{Linear operators}
According to the 
orthogonal splitting
\begin{equation}\label{orthosplitto}
L^{2}(\mathbb{S}^{n};\mathbb{C})=\bigoplus_{k\in \mathbb{N}} E_{k},
\end{equation}
we identify a linear operator 
acting on $L^{2}(\mathbb{S}^{n};\mathbb{C})$
with its matrix 
representation $A:=\Big(A_{[k]}^{[k']}\Big)_{k,k'\in\mathbb{N}}$
in $\mathcal{L}(h^{0})$ (recall \eqref{spseq})
with blocks $A_{[k]}^{[k']}\in \mathcal{L}(E_{k'};E_{k})$.
Notice that each block
$A_{[k]}^{[k']}$ is a $d_{k}\times d_{k'}$ matrix.

\noindent
{\bf Notation.} We shall write 
\begin{equation}\label{notazioneBlock}
A_{[k]}^{[k']}:=\Big(A_{k,j}^{k',j'}\Big)_{\substack{j=1,\ldots, d_{k},\\ j'=1,\ldots,d_{k'}}}\,.
\end{equation}
The action of the operator $A$ on functions $u(x)$ as in \eqref{FouExp} 
of the space variable
in $L^{2}(\mathbb{S}^{n};\mathbb{C})$ is given by
\[
(Au)(x)=\sum_{k\in \mathbb{N}}(Az)_{[k]}\cdot\Phi_{[k]}(x)\,,\qquad z_{[k]}\in \mathbb{C}^{d_{k}}\,,\qquad 
(Az)_{[k]}=\sum_{j\in \mathbb{N}}A_{[k]}^{[j]}z_{[j]}\,.
\]

\vspace{0.9em}
\noindent
{\bf Time-dependent matrices.}
In this paper we also consider 
 $\vphi$-dependent families of linear operators 
\begin{equation}\label{timeMat}
\mathbb{T}^{d}_{\s}\ni\vphi\mapsto A=A(\vphi)=\sum_{l\in\mathbb{Z}^{d}}A(l)e^{\ii l\cdot\vphi}\in \mathcal{L}(h^{0})
\end{equation}
where $A(l)\in \mathcal{L}(h^{0})$,
for any $ l\in \mathbb{Z}^{d}$\,.
We also regard $A$ as an operator acting on functions 
$u(\vphi,x)$ of space-time (see \eqref{timefunc})
as
\[
(Au)(\vphi,x)=(A(\vphi)u(\vphi,\cdot))(x)\,.
\]
More precisely, expanding $u$ as in \eqref{decompogo},
we have
\begin{equation}\label{azione}
(Au)(\vphi,x)=\sum_{l\in \mathbb{Z}^{d},k\in \mathbb{N}}(Az)_{[k]}(l)e^{\ii l\cdot\vphi}\Phi_{[k]}(x)\,,
\qquad (Az)_{[k]}(l)=\sum_{p\in \mathbb{Z}^{d},k'\in N}A_{[k]}^{[k']}(l-p)z_{[k']}(p)\,.
\end{equation}
On operators as in \eqref{timeMat} we define the following norm.

\begin{definition}{\bf ($(s,\s)$-decay norm)}\label{decayNorm}
We define the $(s,\s)$-decay norm of a matrix $A$  in \eqref{timeMat}
as
\begin{equation}\label{decayNorm2}
|A|_{s,\s}^{2}:=\sum_{l\in \mathbb{Z}^{d},h\in \mathbb{N}}\langle l,h\rangle^{2s}
e^{2|l|\s}\sup_{|k-k'|=h}\|A_{[k]}^{[k']}(l)\|^{2}_{\mathcal{L}(L^{2})}
\end{equation}
where $\|\cdot\|_{\mathcal{L}(L^{2})}$ is the $L^{2}$-operator norm 
in $\mathcal{L}(E_{k'},E_k)$.
We denote by $\mathcal{M}_{s,\s}$ the space of matrices of the form \eqref{timeMat} with finite $(s,\s)$-decay norm.

Consider a family 
$\mathcal{O}\ni\omega\mapsto A(\omega)\in \mathcal{M}_{s,\s}$
where $\mathcal{O}$ is a compact subset of $\mathbb{R}^{d}$, $d\geq 1$.
For $\gamma>0$ we define the Lipschitz decay norm
as
\begin{equation}\label{decayLip}
|A|^{\gamma,\mathcal{O}}_{s,\s}:=|A|^{sup,\mathcal{O}}_{s,\s}
+\gamma
|A|^{lip,\mathcal{O}}_{s,\s}=
\sup_{\omega\in \mathcal{O}}|A(\omega)|_{s,\s}+
\gamma
\sup_{\substack{\omega_1,\omega_2\in \mathcal{O}\\ 
\omega_1\neq\omega_2}}\frac{|A(\omega_1)-A(\omega_2)|_{s,\s}}{|\omega_1-\omega_2|}\,.
\end{equation}
We denote by $\mathcal{M}^{\gamma,\mathcal{O}}_{s,\s}$ the space of 
families of matrices $A(\omega)$
with finite  $|\cdot|_{s,\s}^{\gamma,\mathcal{O}}$-norm.
\end{definition}
For the properties of the $(s,\s)$-decay norm we refer the reader to
Lemma 
\ref{DecayAlg} in Appendix \ref{techtech}.
\begin{remark}
Notice that, if the $(s,\s)$-decay norm of a matrix $A$ is finite, then
\[
\|A_{[k]}^{[k']}\|_{\mathcal{L}(L^{2})}\leq C(s) |A|_{s,\s} \langle k-k'\rangle^{-s}\,.
\]
\end{remark}
We deal with a larger class of linear operators.
\begin{definition}
\label{1smooth}
Define the diagonal $\vphi$-independent operator  $\mathcal{D}$, 
acting on sequences $z\in \ell_{0,\s}$ (see \eqref{timeseq}), as 
(recall \eqref{eigen})
\begin{equation}\label{Diag}
\mathcal{D}z:=
{\rm diag}_{l\in \mathbb{Z}^{d}, k\in 
\mathbb{N}}\big(\lambda_{k}^{\frac{1}{2}}\big)z=
\big(\lambda_{k}^{\frac{1}{2}} z_{[k]}(l)\big)_{l\in \mathbb{Z}^{d}, k\in \mathbb{N}}\,.
\end{equation}
For $\beta\in \mathbb{R}$
we define the  norm $\bral \cdot \brar_{\beta,s,\s}$ 
of a matrix $A$  in \eqref{timeMat}
as
\begin{equation}\label{decayNorm3}
\bral A\brar_{\beta, s,\s}:=|\mathcal{D}^{-\beta}A|_{s,\s}
+|A\mathcal{D}^{-\beta}|_{s,\s}\,.
\end{equation}
We denote by $\mathcal{M}_{\beta,s,\s}$ the space of matrices of the form \eqref{timeMat} with finite $\bral \cdot\brar_{\beta,s,\s}$-norm.

Consider a family 
$\mathcal{O}\ni\omega\mapsto A(\omega)\in \mathcal{M}_{\beta,s,\s}$
where $\mathcal{O}$ is a compact subset of $\mathbb{R}^{d}$, $d\geq 1$.
For $\gamma>0$ we define the Lipschitz norm
as
\begin{equation}\label{decayLip33}
\begin{aligned}
\bral A\brar^{\gamma,\mathcal{O}}_{\beta,s,\s}&:=
\bral A\brar^{sup,\mathcal{O}}_{\beta,s,\s}+\gamma
\bral A\brar^{lip,\mathcal{O}}_{\beta,s,\s}
=
\sup_{\omega\in \mathcal{O}}\bral A(\omega)\brar_{\beta,s,\s}+
\gamma
\sup_{\substack{\omega_1,\omega_2\in \mathcal{O}\\ \omega_1\neq\omega_2}}\frac{\bral A(\omega_1)-A(\omega_2)\brar_{\beta,s,\s}}{|\omega_1-\omega_2|}\,.
\end{aligned}
\end{equation}
We denote by $\mathcal{M}^{\gamma,\mathcal{O}}_{\beta,s,\s}$ 
the space of 
families of matrices $A(\omega)$
with finite  $\bral \cdot\brar_{\beta,s,\s}^{\gamma,\mathcal{O}}$-norm.
If $\beta<0$ we say that $A\in \mathcal{M}^{\gamma,\mathcal{O}}_{\beta,s,\s}$
is a $\beta$-smoothing operator. 
If 
$A\in\mathcal{M}^{\gamma,\mathcal{O}}_{\beta,s,\s}$ 
does not depend on $\vphi$ we simply write 
$A\in \mathcal{M}^{\gamma,\mathcal{O}}_{\beta,s}$.
\end{definition}

 \begin{remark}\label{inclusioni}
 We have the following simple inclusions
 for $\beta'>\beta$ and $\nu_1,\nu_{2}\geq 0$:
 \[
 \mathcal{M}^{\gamma,\mathcal{O}}_{\beta,s,\s}\subset
 \mathcal{M}^{\gamma,\mathcal{O}}_{\beta',s,\s}\,,\qquad
  \mathcal{M}^{\gamma,\mathcal{O}}_{\beta,s+\nu_{1},\s+\nu_{2}}\subset
 \mathcal{M}^{\gamma,\mathcal{O}}_{\beta,s,\s}\,.
 \]
 The inclusions are continuous. For further properties of
 the operators of Def. \ref{1smooth} we refer to Appendix \ref{techtech}.
 \end{remark}

\subsection{Hamiltonian structure}
In this subsection we introduce a special class of linear operators.

\begin{definition}
\label{Hermo}
Consider a linear operator $M\in \mathcal{L}(h^{0})$
and a  family of maps 
$\vphi\mapsto A(\vphi)$ in $\mathcal{M}_{0,\s}$.

\noindent
$\bullet$ {\bf (Hermitian operators)}. We say that $M$ is  \emph{Hermitian} if
\begin{equation}\label{albero}
M_{[k]}^{[k']}=\Big(M_{k,m_k}^{k',m_{k'}}\Big)_{\substack{m_{k}=1,\ldots, d_{k} \\ 
m_{k'}=1,\ldots,d_{k'}}}\qquad{\rm is\,\; such\,\; that}\qquad 
M_{k,m_k}^{k',m_{k'}}=\ov{M_{k',m_k'}^{k,m_{k}}}
\end{equation}
for any $k,k'\in\mathbb{N}$.
To lighten the notation we shall also write that
$M_{[k]}^{[k']}=\ov{M_{[k']}^{[k]}}$ instead of the \eqref{albero}.
 We say that $A(\vphi)$ is Hermitian if and only if
\begin{equation*}
A_{[k]}^{[k']}(l)=\ov{A_{[k']}^{[k]}(-l)}\,, \qquad \forall \,l\in \mathbb{Z}^{d}\,,\; k,k'\in \mathbb{N}\,.
\end{equation*}

\noindent
$\bullet$ {\bf (Hamiltonian operators)}. We say that $M$ is  
\emph{Hamiltonian}
 if \,$\ii M$ is Hermitian. 
We say that  $A(\vphi)$
 is Hamiltonian if and only if
 \begin{equation}\label{albero3}
A_{[k]}^{[k']}(l)=-\ov{A_{[k']}^{[k]}(-l)}\,, 
\qquad \forall \,l\in \mathbb{Z}^{d}\,,\; 
k,k'\in \mathbb{N}\,.
\end{equation}

\noindent
$\bullet$ {\bf (Block-diagonal operators)}. 
We say that $A(\vphi)$
 is  \emph{block-diagonal}  if  and only if $A_{[k]}^{[k']}(\vphi)=0$ 
 for any $k\neq k'$ and any $\vphi\in \mathbb{T}^{d}_{\s}$.
\end{definition}

\begin{definition}{\bf (Normal form)}\label{normalform200}
We say that a matrix $M$ is in 
 \emph{normal form} 
 if it is $\vphi$-independent, Hermitian and block-diagonal 
 according to Definition \ref{Hermo}.
 Given a Hermitian  family of maps 
 $\vphi\mapsto A(\vphi)$ in $\mathcal{M}_{0,\s}$
 we define its normal form
${\rm Diag}A =\big(({\rm Diag}A)_{[k]}^{[k']}(l)
\big)_{l\in\mathbb{Z}^{d},k,k'\in\mathbb{N}}$
as
\begin{equation}\label{newNorm}
({\rm Diag}A)_{[k]}^{[k]}(0):= A_{[k]}^{[k]}(0)\,, 
\qquad ({\rm Diag}A)_{[k]}^{[k']}(l):=0\, \text{ for } l\neq0\,, k,k'\in \mathbb{N}\,,\quad
{\rm or }\quad l=0\,, k\neq k'\,.
\end{equation}
\end{definition}

\noindent
Let $\omega\cdot\pa_{\vphi}$ be the diagonal operator acting on sequences $z\in \ell_{0,\s}$ (see \eqref{timeseq}) defined by
\begin{equation}\label{omegaphi}
\omega\cdot\pa_{\vphi}z:={\rm diag}_{l\in \mathbb{Z}^{d}, k\in \mathbb{N}}(\ii \omega\cdot l)z=
(\ii \omega\cdot l z_{[k]}(l))_{l\in \mathbb{Z}^{d}, k\in \mathbb{N}}.
\end{equation}
This operator is Hamiltonian and thus 	an operator of the form $\omega\cdot\pa_{\vphi}+M(\vphi)$
is Hamiltonian 
 if and only if $M(\vphi)$ is Hamiltonian.

\vspace{0.9em}
\noindent
{\bf Conjugation under Hamiltonian flows.}
Consider the operator 
\begin{equation}\label{proto}
L(\vphi):=\omega\cdot\pa_{\vphi}+M\,,
\end{equation}
where  $\omega\cdot\pa_{\vphi}$ is defined in \eqref{omegaphi},
the operator $M=M(\vphi)\in \mathcal{M}_{0,\s}$ is Hamiltonian 
(see Def. \ref{Hermo}).
We shall study how the operator 
$L(\vphi)$ changes under the map 
\begin{equation}\label{exp}
\Phi:=e^{\ii A}:=\sum_{p\geq0}\frac{1}{p!}(\ii A)^{p}\,,\end{equation}
for some $A\in \mathcal{M}_{0,\s}$ Hermitian.
For the well-posedness of a map of the form \eqref{exp}
we refer to Lemma \ref{well-well} in Appendix \ref{techtech}.
By using Lie expansion the conjugate operator
$M^{+}=M^{+}(\vphi):=e^{\ii A}M(\vphi)e^{-\ii A}$
 has the form
\begin{equation}\label{lie1}
M^{+}=M^{+}(\vphi):=e^{\ii A}M(\vphi)e^{-\ii A}=\sum_{p\geq0}\frac{1}{p!}{\rm ad}^{p}_{\ii A}(M)\,,
\end{equation}
where
\begin{equation}\label{commu}
{\rm ad}^{0}_{\ii A}(M)=M\,,\quad {\rm ad}^{p}_{\ii A}(M)=
{\rm ad}^{p-1}_{\ii A}([\ii A, M])\,,\quad 
[\ii A, M]=\ii A M-\ii M A\,.
\end{equation}
Using the \eqref{lie1} we also deduce that (recall \eqref{omegaphi})
\begin{equation}\label{lie2}
e^{\ii A}\omega\cdot\pa_{\vphi}e^{-\ii A}=\omega\cdot\pa_{\vphi}+
\widetilde{M}^{+}(\vphi)=\omega\cdot\pa_{\vphi}
-\ii \omega\cdot\pa_{\vphi}A
-\sum_{p\geq2}\frac{1}{p!}{\rm ad}^{p-1}_{\ii A}
(\ii \omega\cdot\pa_{\vphi}A)\,.
\end{equation}

\begin{lemma}\label{conjHam}
If $M$ and $\ii A$  are Hamiltonian linear operators
 then
 $M^{+}$ and $\widetilde{M}^{+}$ in \eqref{lie1} and \eqref{lie2}
 are Hamiltonian.
\end{lemma}
\begin{proof}
To prove the lemma it is sufficient to check that 
$[\ii A,M]$ and $\ii \omega\cdot\pa_{\vphi}A$ are Hamiltonian.
We have that, for any $k,k'\in \mathbb{N}, l\in \mathbb{Z}^{d}$
\[
\big([\ii A,M]\big)_{[k]}^{[k']}(l)=\ii
\sum_{p\in\mathbb{Z}^{d},j\in\mathbb{N}}A_{[k]}^{[j]}(l-p)M_{[j]}^{[k']}(p)
-M_{[k]}^{[j]}(l-p)A_{[j]}^{[k']}(p)\,.
\]
Hence the claim follows using that $\ii A$ and $M$ are Hamiltonian, i.e. their coefficients satisfy \eqref{albero3}. Reasoning similarly one deduces the claim for $\ii \omega\cdot\pa_{\vphi}A$.
\end{proof}
Notice that in view of Lemma \ref{conjHam} the map  of the form \eqref{exp} with $A$ Hermitian is \emph{symplectic}.

\begin{remark}
Lemma \ref{conjHam} provides only a formal rule of conjugation of
matrices. It does not guarantees that
such conjugate is a bounded operator on the spaces
$\ell_{s,\s}$ with $s\gg1$. The key information is that
(at least formally) the flow $\Phi$ of a Hamiltonian operator (see \eqref{exp})
preserves the Hamiltonian structure, i.e. the map $\Phi$ is symplectic.
\end{remark}

\section{An abstract reducibility result and its application to 
(\ref{nls})}\label{S.main1}

In this section we state our main abstract result 
and we give some applications for the Schr\"odinger equation on spheres.
\subsection{Abstract reducibility result}\label{S.main}
Fix the parameters $s>(d+n)/2$, $\sigma>0$, $\gamma>0$ 
as in the previous sections and let us add three new parameters
\begin{equation}\label{parametrix}
0\leq \alpha<\frac{1}{2}\,,\quad \beta:=1-2\alpha>0\,,
\quad \nu\geq \alpha+\beta=1-\alpha\,.
\end{equation}
Consider (recall Def. \ref{decayNorm}, \ref{1smooth}) 
an operator of the form
\begin{equation}\label{start0}
G=G(\vphi)=G(\omega;\vphi):=
\omega\cdot\pa_{\vphi}
-\ii\mathcal{D}^{2}+R+R'\,, \qquad 
R\in \mathcal{M}^{\gamma,\mathcal{O}}_{\alpha,s+\nu,\sigma}\,,
R'\in \mathcal{M}^{\gamma,\mathcal{O}}_{-\beta,s,\sigma}
\end{equation}
where  $\omega\cdot\pa_{\vphi}$ and $\mathcal{D}$
are defined respectively in \eqref{omegaphi} and \eqref{Diag}
and $\mathcal{O}$ is a compact
subset of $\mathbb{R}^{d}$.
Assume also that $R$ and $R'$ are Hamiltonian according to Definition 
\ref{Hermo}
and that $R$ is diagonal free i.e.
\begin{equation}\label{ipoBlocco}
R_{[k]}^{[k]}(\vphi)=0\,,\qquad \forall\, k\in\mathbb{N}\,,
\quad \vphi\in\mathbb{Z}^{d}\,.
\end{equation}
We notice that $R$ is unbounded while $R'$ is smoothing.

\begin{theorem}{\bf (Reducibility)}\label{thm:redu}
Let $\gamma>0$. There exist  $\epsilon_0>0$ and $C>0$ depending only
on $s,d,n,\alpha$ such that, if 
\begin{equation}\label{pescerosso}
\epsilon:=\gamma^{-1}
(\bral R\brar_{\alpha,s+\nu,\s}^{\gamma,\mathcal{O}}
+\bral R'\brar_{-\b,s,\s}^{\mu,\mathcal{O}}) 
\qquad \text{ satisfies } 
\qquad \epsilon< \epsilon_0
\end{equation}
then the following holds.
There exist:

\vspace{0.5em}
\noindent
(i) {\bf (Cantor set)} A cantor set $\mathcal{O}_{\infty}\subset \mathcal{O}$
such that
\begin{equation}\label{measureEst}
{\rm meas}(\mathcal{O}\setminus\mathcal{O}_{\infty})\leq 
 C \gamma \, ;
\end{equation}

\noindent
(ii) {\bf (Normal form)}
an operator
$\mathcal{Z}\in 
\mathcal{M}^{\gamma,\mathcal{O}_{\infty}}_{-\beta,s}$ in 
normal form (see Def. \ref{normalform200})
satisfying
\[
\bral\mathcal{Z}\brar_{-\beta,s}^{\gamma,\mathcal{O}_{\infty}}
\leq C\epsilon\gamma\, ,
\]
and the eigenvalues of the block
$\mathcal{Z}_{[k]}^{[k]}$, denoted $\mu^{(\infty)}_{k,j}$, 
 $j=1,\ldots,d_{k}$, are  Lipschitz functions from $ \mathcal{O}$ into $\R$,
and satisfy
\begin{equation}\label{stimaeigen}
\sup_{\substack{k\in\mathbb{N}\\j=1,\cdots,d_k}}
\langle k\rangle^{\beta} 
|\mu^{(\infty)}_{k,j}|^{\gamma,\mathcal{O}}\leq
C\epsilon \gamma\,;
\end{equation}

\noindent
(iii) {\bf (Conjugacy)}  
 A Lipschitz family of invertible and symplectic maps 
$\Phi=\Phi(\omega) : \ell_{s,\s/4}^{\gamma, \mathcal{O}_{\infty}}
\to  \ell_{s,\s/4}^{\gamma, \mathcal{O}_{\infty}}$, 
of the form $\Phi={\rm Id}+\Psi$ satisfying 
\begin{align}
&\bral \Phi^{\pm1}
-{\rm Id}\brar_{-\beta,s,\s/4}^{\gamma,\mathcal{O}_{\infty}}
\leq C(s)\epsilon \,,\label{stimeTRA}\\
&\|\Phi^{\pm1}(\omega;\vphi)-{\rm Id}\|_{\mathcal{L}(h^{s}; h^{s+\beta})}\leq C(s,\s,\s')\epsilon\,\qquad \forall\; \om\in\mathcal{O}\,, \forall\; \vphi\in\mathbb{T}^{d}_{\s'}\,,\;\;\s'<\frac{\s}{4} 
\label{stimeTRA1000}
\end{align}
such that, for any $\o\in \mathcal{O}_{\infty}$,
\begin{equation}\label{coniugato}
L^{(\infty)}:=\Phi(\omega)\circ G\circ\Phi^{-1}(\omega)=\omega\cdot\pa_{\vphi}
-\ii \mathcal{D}^{2}-\ii \mathcal{Z}\,.
\end{equation}
\end{theorem}

Theorem \ref{thm:main} will be proved in sections \ref{regu}, \ref{sec:itero}

\subsection{Application to (\ref{nls}) on the sphere}\label{AppliNLS}

In this section we consider a more general setting than in introduction. In fact we consider the Schr\"odinger equation
\[\label{nls2}\tag{LS2}
\ii \pa_{t}u=\Delta u+\eps\big( \Rc(\o t,x)+\Rc'(\omega t,x)\big)u\,,
\qquad u=u(t,x)\,,
\quad t\in \mathbb{R}\,,\quad 
x\in \mathbb{S}^{n}\,,\quad n\geq1\,,
\]
where $\Delta$ denotes the Laplace-Beltrami operator on 
$\mathbb{S}^{n}$ and $\Rc$ and $\Rc'$
 are time-dependent families of linear operators corresponding, 
 in their matrix representation with respect 
 to the spherical harmonics basis, 
 to Hamiltonian matrices 
 $R\in \mathcal{M}^{\g,\mathcal{O}_0}_{\alpha,s+\nu,\sigma}$, 
 $R'\in \mathcal{M}^{\g,\mathcal{O}_0}_{-\beta,s,\sigma}$ 
 with $R$ diagonal free   as in \eqref{parametrix}, \eqref{start0}, 
 \eqref{ipoBlocco}. Let us  choose 
 $\gamma=\eps^\delta$ for some $0<\delta<1$. 
 The assumption \eqref{pescerosso} reads 
 $\eps<\eps_0$ with 
 $\eps_0^{1-\delta}=(\bral R\brar_{\alpha,s+\nu,\s}^{\g,\mathcal{O}_0}
 +\bral R'\brar_{-\b,s,\s}^{\g,\mathcal{O}_0} )^{-1}\epsilon_0$.
  So we have the following.
\begin{theorem}\label{thm:main2}  Let $0<\delta<1$, 
$0\leq \alpha<1/2$ and  $s>(d+n)/2$. 
There exists  $\eps_0>0$ such that, for any $0<\eps\leq \eps_0$ there is
 a set $\mathcal{O}_\eps\subset\mathcal{O}_0\subset\mathbb{R}^{d}$
 with 
 \begin{equation}\label{measureSec3}
 {\rm meas}(\mathcal{O}_0\setminus\mathcal{O}_\eps)\leq \eps^{\delta}
 \end{equation}
 such that the following holds.
 
 \noindent
 For any $\omega\in \mathcal{O}_\e$
 there exist a family of  linear isomorphisms 
 $\Psi(\vphi)\in \mathcal{L}(H^{s}(\mathbb{S}^{n};\mathbb{C}))$,
 analytically 
 depending on $\vphi\in \mathbb{T}^{d}_{\s/2}$ 
 and a block diagonal Hermitian operator $Z\in \mathcal{L}(H^{s}(\mathbb{S}^{n};\mathbb{C});
 H^{s+\beta}(\mathbb{S}^{n};\mathbb{C}))$
satisfying 

$\bullet$ $\Psi(\vphi) $ is unitary on $L^{2}(\mathbb{S}^n;\mathbb{C})$;
 
 $\bullet$ for any $0\leq s'\leq s$
 \begin{equation}\label{stimasec3}
 \|\Psi(\vphi)-{\rm Id}\|_{\mathcal{L}(H^{s'},H^{s'+\beta})}+
  \|\Psi(\vphi)^{-1}-{\rm Id}\|_{\mathcal{L}(H^{s'},H^{s'+\beta})}
  \leq \e^{1-\delta}\,,
 \end{equation}
 
 $\bullet$
 the function $t\mapsto u(t,\cdot)\in H^{s'}(\mathbb{S}^{n};\mathbb{C})$
 solves \eqref{nls2} if and only if the map
 $t\mapsto v(t,\cdot):=\Psi(\omega t)u(t,\cdot)$ solves the autonomous 
 equation
 \begin{equation}\label{NLSredsec3}
 \ii \pa_t v=\Delta v+\eps Z(v)\,.
 \end{equation}
\end{theorem}

Now it remains to give examples of $\Rc$ and $\Rc'$ that satisfy the right hypothesis. In particular, we need to make sure that \eqref{nls} is in the right framework in such a way Theorem \ref{thm:main} holds true.

 First we verify that a multiplicative potential is an admissible perturbation. 
 \begin{lemma}\label{potok}
Assume that $\vphi\mapsto V(\vphi,\cdot)\in 
H^{s+s_0}(\mathbb{S}^{n};\mathbb{R})$  
analytically extends to $\mathbb{T}^{d}_{\s}$ for some $\s>0$
and with $s_0=s_0(n)$.
Then the matrix that represents the 
multiplication operator by $V$ 
and still denoted by $V$ 
belongs to $\mathcal{M}_{s,\sigma'}$ for any $0<\s'<\s$.
Furthermore if  $V$ is an odd 
function in the space variable then $V$ is diagonal free:
\begin{equation}\label{calo}
V_{[k]}^{[k]}=0\quad  k\in\N\,.
\end{equation}
 \end{lemma}
\begin{proof}
The fact that $V\in\mathcal{M}_{s,\sigma'}$ 
is a consequence of Proposition $2.19$ in \cite{BP} 
(see also Lemma $3.1$ in \cite{BCP}). 
Actually this is the reason why 
we use the  $s$-decay norm 
(see Definition \ref{decayNorm}). 
So we only have to verify the second statement. 
By definition we have
\[
V_{[k]}^{[k']}=(V_{k,j}^{k',\ell})_{\substack{1\leq j\leq d_k\\1\leq \ell\leq d_{k'}}}
\qquad {\rm with }\qquad
V_{k,j}^{k',\ell}:=\int_{\S^n} V(x) \Phi_{k,j}(x)\Phi_{k',\ell}(x) dx\,.
\]
Now the spherical harmonic $\Phi_{k,j}$ 
has the same parity than $k$: 
$\Phi_{k,j}(-x)=(-1)^k\Phi_{k,j}(x)$. 
Therefore, if $V$ is odd, we conclude
\begin{equation}\label{calo3}
V_{[k]}^{[k']}= 0 
\quad \text{if}\quad k+k' \quad \text{even}\,,
\end{equation}
which implies the \eqref{calo}.
\end{proof}

\noindent
Now we consider the perturbation term $W(\o t,x) (-i\partial_\phi)^\alpha$ appearing in \eqref{nls}.
We know that $L_{x_3}=-i\partial_\phi$ also 
diagonalizes in spherical harmonic 
basis\footnote{
Recall that in \eqref{nls} we are in $\S^2$ and 
the spherical harmonic basis is given by 
$\Phi_{k,j}=Ce^{ij\phi}P^j_k(\cos \theta)$ 
for $k\in\N$ and $-k\leq j\leq k$ 
and where $P_k^j$ are the 
Legendre polynomials 
(see for instance wikipedia.org/wiki/Spherical-harmonics).
}:
\[
-i\partial_\phi\Phi_{k,j}=
j\Phi_{k,j}\quad k\in\N,\ \ j=-k,\cdots,k
\]
and we define $(-i\partial_\phi)^\alpha$ by
\begin{equation}\label{partialalfa}
(-i\partial_\phi)^\alpha \Phi_{k,j}={\rm sign}(j)|j|^{\alpha}\Phi_{k,j} 
\quad k\in\N,\ \ j=-k,\cdots,k\,.
\end{equation}

\begin{lemma}\label{potok2}
Assume that $\vphi\mapsto W(\vphi,\cdot)\in 
H^{s+s_0}(\mathbb{S}^{n};\mathbb{R})$  
analytically extends to $\mathbb{T}^{d}_{\s}$ 
for some $\s>0$
and with $s_0=s_0(n,d,\alpha)\geq (d+n)/2$.
 Then the matrix that represents the unbounded 
 operator $R=W(\o t,x) (-i\partial_\phi)^\alpha$  
 belongs to $\mathcal{M}_{\alpha,s+\nu,\sigma'}$
 with $\nu$ as in \eqref{parametrix} and $0<\s'<\s$.
Furthermore if  $W$ is an odd function in the 
space variable then $R$ is diagonal free:
\begin{equation}\label{calo1}
R_{[k]}^{[k]}=0\quad  k\in\N\,.
\end{equation}
 \end{lemma}

\begin{proof}
Since $|j|\leq k$ we have (see \eqref{Diag}) 
$- \mathcal D\leq -i\partial_\phi\leq \mathcal D$ 
in the sense of operators on $\ell_{0,0}$ 
(see \eqref{timeseq}). 
Thus, in view of Definition \ref{1smooth} 
and Lemma \ref{potok}, we get the first part of the Lemma.
It remains we only have to verify the second part. 
By definition we have
\[
R_{[k]}^{[k']}=(R_{k,j}^{k',\ell})_{\substack{-k\leq j\leq k\\-k'\leq \ell\leq {k'}}}
\qquad {\rm with} \qquad
R_{k,j}^{k',\ell}:={\rm sign}(\ell)|\ell|^\alpha
\int_{\S^n} W(x) \Phi_{k,j}(x)\Phi_{k',\ell}(x) dx\,.
\]
So we use again that the spherical harmonic 
$\Phi_{k,j}$ has the same parity than 
$k$ to conclude that if $W$ is odd then
$R_{[k]}^{[k']}$ satisfies \eqref{calo3} and hence \eqref{calo1} holds.
\end{proof}

\begin{proof}[{\bf Proof of Theorem \ref{thm:main}}]
The result follows by Lemmata \ref{potok}, \ref{potok2} and
by Theorem \ref{thm:main2}.
\end{proof}

 We conclude 
 this section with examples of 
 regularizing perturbations 
 $R'\in \mathcal{M}^{\g,\mathcal{O}}_{-\beta,s,\sigma}$. 
 The natural framework is that of pseudo-differential operators. \\
 We denote by
      $S_{\rm cl}^m(\S^n)$ the space of classical real valued symbols of
      order $m\in \R$ on the cotangent $T^*(S^n)$ of $\S^n$ (see
      H\"ormander \cite{ho} for more details).
      \begin{definition}
        \label{pseudo}
We say that $A\in\mathcal{P}_m$ 
if it is a pseudodifferential operator 
(in the sense of H\"ormander \cite{ho}, see also  \cite{BGMR2} ) with
symbol of class $S^m_{\rm cl}(M)$.
      \end{definition}

 We have
 \begin{lemma}\label{pseudoLemma}
Let $\beta>0$ and assume that $\vphi\mapsto R(\vphi,\cdot)\in 
\mathcal{P}_\beta$  analytically extends to 
$\mathbb{T}^{d}_{\s}$ for some $\s>0$.
Then the matrix that represents the  operator $R$  
belongs to $\mathcal{M}_{-\beta,s,\sigma}$ for all $s>(n+d)/2$.
 \end{lemma}
 
\begin{proof}
We use the so called commutator Lemma: 
Let A be a linear operator which maps 
$H^s(\S^n)$ into itself and define the sequence of operators
\[
A_0:=A\,, \quad  A_N=[(-\Delta)^{1/2},A_{N-1}]\,,\quad N\geq1\,,
\]
we have for any $\Phi_k\in E_k$, $\Phi_{k'}\in E_{k'}$
\begin{equation}\label{cmmu}
|\langle A\Phi_k,\Phi_{k'}\rangle|\leq 
\frac{1}{|\lambda_k^{1/2}-\lambda_{k'}^{1/2}|^N}| 
\langle A_N\Phi_k,\Phi_{k'}\rangle|\,.
\end{equation}
Consider the operator $A:=D^\beta R$, 
by hypothesis $A\in\mathcal{A}_0$ 
and $(-\Delta)^{1/2}\in\mathcal{A}_1$ 
so by the fundamental property of pseudo-differential 
operators we deduce that for all $N\geq1$,  $A_N\in\mathcal{A}_0$. 
As a consequence $ \|A_N\Phi_k\|\leq C_N \|\Phi_k\|$ 
and thus by \eqref{commu}
$$\|A_{[k]}^{[k']}\|\leq  \frac{C_N}{|k-k'|^N}.$$
Taking $N=N(s)$ large enough we 
deduce that $A\in\mathcal M_{s,\s'}$ and 
thus $A\in \mathcal M_{-\beta,s,\s'}$.
\end{proof}

\section{The regularization step}\label{regu}
In this section we show that Theorem \ref{thm:main} (where $R$ is unbounded) can be reduced to a reducibility problem with a smoothing perturbation.
To do this, we use the properties of the eigenvalues
of the Laplacian operator on the spheres to show that the operator
$G$ in \eqref{start0} can be conjugated to a diagonal operator
plus a smoothing remainder.
More precisely we have the following (We use the same set of constants as in the section \ref{S.main}).

\begin{proposition}\label{riduco}
There exists $\bar{\e}>0$ and $C>0$ (depending only on $s,n,d$) 
such that for any 
$0<\e\leq\bar{\e}$, if    $R$ as in \eqref{start0}
satisfies 
\begin{equation}\label{tuttopiccolo}
\bral R\brar_{\alpha,s+\nu,\s}^{\gamma,\mathcal{O}}\leq \e
\end{equation}
 then the following holds.
There exists 
 a Lipschitz family of invertible and symplectic maps 
map $\mathcal{T}:=\mathcal{T}(\omega):={\rm Id}+\mathcal{F}$,
with $\mathcal{F}\in \mathcal{M}_{\alpha-1,s+\nu,\s}^{\gamma,\mathcal{O}}$
and (see Def. \ref{1smooth})
\begin{equation}\label{bracket}
\bral \mathcal{F}\brar_{\alpha-1,s+\nu,\s}^{\gamma,\mathcal{O}}
\leq C\e\,,
\end{equation}
such that the conjugate of the operator $G$ in \eqref{start0} has the form
\begin{equation}\label{conj}
\mathcal{T}\circ G\circ \mathcal{T}^{-1}=
\omega\cdot\pa_{\vphi}
-\ii(\mathcal{D}^{2}+Z)+M
\end{equation}
where $Z\in \mathcal{M}_{-\beta,s}^{\gamma,\mathcal{O}}$ is in normal form
(see Def. \ref{normalform200}) and 
$M\in \mathcal{M}_{-\beta,s,\s/2}^{\gamma,\mathcal{O}}$
is   Hamiltonian (see Def. \ref{Hermo})
 and satisfy
\begin{equation}\label{conj2}
\bral Z\brar_{-\beta,s}^{\gamma,\mathcal{O}},\ 
\bral M\brar_{-\beta,s,\s/2}^{\gamma,\mathcal{O}}\leq 
C(\bral R\brar_{\alpha,s+\nu,\s}^{\gamma,\mathcal{O}}
+\bral R'\brar_{-\beta,s,\s}^{\gamma,\mathcal{O}})\,.
\end{equation}
Finally $M$ is such that $M_{[k]}^{[k]}(0)=0$ for any $k\in\mathbb{N} $.
\end{proposition}

\begin{proof}
Consider the matrix 
\begin{equation}\label{genero}
\mathcal{A}=\Big(\mathcal{A}_{[k]}^{[k']}(l)\Big)_{\substack{l\in \mathbb{Z}^{d}\\
k,k'\in \mathbb{N}}}\,, \qquad
\mathcal{A}_{[k]}^{[k']}(l):=\left\{
\begin{aligned}
& \frac{\ii {R}_{[k]}^{[k']}(l)}{\lambda_{k}-\lambda_{k'}}\,,
\quad \forall\,l\in \mathbb{Z}^{d}\,\;\;\; k,k'\in\mathbb{N}\, \;\;k\neq k'\,,\\
&0 \qquad\qquad k=k'
\end{aligned}
\right.
\end{equation}
with $\lambda_{k}$ defined in \eqref{eigen}. Since $R$ is Hamiltonian one verifies that
$\mathcal{A}$ is Hamiltonian. Moreover, using that, for $k\neq k'$,
one has $|\lambda_{k}-\lambda_{k'}|\geq k+k'$, we deduce that
(recall \eqref{Diag})
\[
|\mathcal{D}^{1-\alpha}\mathcal{A}|_{s+\nu,\s}^{2}\leq 
\sum_{l\in \mathbb{Z}^{d},h\in \mathbb{N}}\langle l,h\rangle^{2(s+\nu)}
e^{2|l|\s}\sup_{|k-k'|=h}\|\lambda_{k}^{\frac{1-\alpha}{2}}
\mathcal{A}_{[k]}^{[k']}(l)\|^{2}_{\mathcal{L}(L^{2})}
\leq \sup_{k\in \mathbb{N}}\Big(\frac{\lambda_{k}^{\frac{1}{2}}}{k+k'}\Big)^{2}
|R|^{2}_{\alpha,s+\nu,\s}\,.
\]
Reasoning in a similar way for $\mathcal{A}\mathcal{D}^{1-\alpha}$ 
one obtain
\begin{equation}\label{stimasucalA}
\bral \mathcal{A}\brar_{\alpha-1,s+\nu,\s}^{\gamma,\mathcal{O}}\leq C 
\bral R\brar_{\alpha,s+\nu,\s}^{\gamma,\mathcal{O}}
\end{equation}
for some $C=C(s,n)>0$.
We set $\mathcal{T}:={\rm Id}+\mathcal{F}:=e^{\mathcal{A}}$ which has the form
\eqref{exp} with $\ii A\rightsquigarrow \mathcal{A}$.
Estimates  \eqref{stimasucalA}, \eqref{tuttopiccolo} implies 
\eqref{mito30} for $\e$ small enough. Hence
the bound \eqref{bracket} follows by Lemma
\ref{well-well}.
By \eqref{genero} and the hypothesis \eqref{ipoBlocco}
we have that 
\begin{equation}\label{omo100}
R+\big[\mathcal{A}, -\ii \mathcal{D}^{2}\big]=0\,.
\end{equation}
Thus
formul\ae \, \eqref{lie1},  \eqref{lie2} and  \eqref{omo100}
imply that 
 that $\mathcal{T}\circ G\circ \mathcal{T}^{-1}$
has the form \eqref{conj} with
\begin{equation}\label{restoresto}
-\ii Z+M:=
-\omega\cdot\pa_{\vphi}\mathcal{A}
+\big[\mathcal{A},R\big]+
\sum_{p\geq0}\frac{1}{p!}{\rm ad}^{p}_{\mathcal{A}}\big(R'\big)
-
\sum_{p\geq2}\frac{1}{p!}{\rm ad}^{p-1}_{\mathcal{A}}
(\ii \omega\cdot\pa_{\vphi}\mathcal{A})\,.
\end{equation}
We define $-\ii Z$ as the 
normal form (see \eqref{newNorm} in Def. \ref{normalform200})
of the 
previous expression while $M$ is defined by difference.
Let $0<\s_{+}<\s$. Then
we have
\[
\begin{aligned}
|\mathcal{D}^{1-\alpha}\omega\cdot\pa_{\vphi}\mathcal{A}|^{2}_{s+\nu,\s_{+}}
&\leq
\sum_{l\in \mathbb{Z}^{d},h\in \mathbb{N}}\langle l,h\rangle^{2s}
e^{2|l|\s}\sup_{|k-k'|=h}\|\lambda_{k}^{\frac{1-\alpha}{2}}
\mathcal{A}_{[k]}^{[k']}(l)\|^{2}_{\mathcal{L}(L^{2})}e^{-2(\s-\s_{+})|l|}
|l|^{2}
\\
&\leq (\s-\s_{+})^{-2}
\bral \mathcal{A}\brar_{\alpha-1,s+\nu,\s}^{\gamma,\mathcal{O}}\,.
\end{aligned}
\]
With a similar reasoning one concludes
\begin{equation}\label{stimasucalA2}
\bral\omega\cdot\pa_{\vphi}\mathcal{A}
\brar^{\gamma,\mathcal{O}}_{\alpha-1, s+\nu,\s_{+}}
\stackrel{\eqref{stimasucalA}}{\leq_{s}}(\s-\s_{+})^{-1}
\bral R\brar_{\alpha, s+\nu,\s}^{\gamma,\mathcal{O}}\,.
\end{equation}
By estimate \eqref{stim2} in Lemma \ref{DecayAlg2}  and \eqref{parametrix} 
we also obtain
\begin{equation}\label{stimasucalA3}
\bral \big[\mathcal{A}, R\big]
\brar^{\gamma,\mathcal{O}}_{-\beta,s,\s_{+}}\stackrel{\eqref{stimasucalA}}{\leq_{s}}
\Big(\bral R
\brar^{\gamma,\mathcal{O}}_{\alpha,s+\nu,\s}\Big)^{2}\,.
\end{equation} 
The \eqref{conj2} follows by  
using the smallness condition \eqref{tuttopiccolo}, the estimates
 \eqref{stimasucalA}, \eqref{stimasucalA2}, \eqref{stimasucalA3}, 
 and reasoning as in Lemma \ref{well-well}.
Finally the operator $M$ is Hamiltonian by Lemma \ref{conjHam}.
\end{proof}

\section{The iterative reducibility scheme}\label{sec:itero}

In this section we prove Theorem \ref{thm:redu} taking into account the regularization step given in Proposition \ref{riduco}. This mean that we show how to block-diagonalize the operator 
 \begin{equation}\label{start}
L=L(\omega;\vphi):=\omega\cdot\pa_{\vphi}-
\ii (\mathcal{D}^{2}+Z)+M
\end{equation}
with $Z\in  \mathcal{M}_{-\beta,s}^{\gamma,\mathcal{O}}$
in normal form and 
$M\in \mathcal{M}_{-\beta,s,\s}^{\gamma,\mathcal{O}}$
 Hamiltonian 
 satisfying that
\begin{equation}\label{small}
\Theta_0:=\gamma^{-1}\bral Z\brar_{-\beta,s,\s}^{\gamma,\mathcal{O}}\,,\quad \eps_0:=\gamma^{-1}\bral M\brar_{-\beta,s,\s}^{\gamma,\mathcal{O}} \end{equation}
are small enough. Actually at the beginning of our iterative process we can take $\Theta_0=\eps_0$ (see \eqref{conj2}) but during the process it will be important to distinguish between the size of the normal form (which essentially will not change) and the size of the remainder term (which will converge rapidly to zero).
Consider the diophantine set
\begin{equation}\label{zeroMel}
\mathcal{G}_0:=\big\{\omega\in [1/2,3/2]^{d} \,:\, |\omega\cdot l|\geq 
\frac{4\gamma}{|l|^{\tau_0}}\,, \; \forall \, l\in\mathbb{Z}^{d}\big\}\,,
\qquad \tau_0:= d+1
\end{equation}
We remark that it is know that ${\rm meas}(\mathcal{G}_0)\lesssim\gamma$.
In the following we shall assume that the set of parameters
$\mathcal{O}$ satisfies $\mathcal{O}\subseteq \mathcal{G}_0$.

\subsection{ KAM strategy}\label{KAMstep}
We begin with $L$ given by \eqref{start}, we seach for $\Phi=e^{S}$ a canonical change of variable such that
\begin{equation}\label{goal}L^+=L^+(\vphi,\omega):=\Phi\circ L \circ\Phi^{-1}=
\omega\cdot\pa_{\vphi}-\ii \big(\mathcal{D}^{2}+Z_+\big)+M_+
\end{equation}
where $Z_+$ is block-diagonal and $\vphi$-independent, 
$N_+=\mathcal{D}^{2}+Z_+$ is the new normal form, 
$\eps$ close to $N_0=\mathcal{D}^2$ and  
the new perturbation $M_+$ is 
expected of size $O (\eps^2)$.\\
Using the expansion \eqref{lie1}, 
\eqref{lie2} with $\ii A\rightsquigarrow S$
we have that
\begin{align}\label{L+}
L^{+}&=\omega\cdot \pa_{\vphi}-\ii \big(\mathcal{D}^{2}+Z\big)+M
-\omega\cdot\pa_{\vphi}S+\ii\big[\mathcal{D}^{2}+Z,S\big]
\nonumber\\&\qquad +\sum_{p\geq 2}\frac{1}{p!}{\rm ad}_{S}^{p-1}
\Big(
-\omega\cdot\pa_{\vphi}S+\ii \big[\mathcal{D}^{2}+Z,S\big]
\Big)+
\sum_{p\geq1}\frac{1}{p!}{\rm ad}_{S}^{p}\big(M\big)\,.
\end{align}
Formally, if we are able to construct $S=O(\eps)$ satisfying the the so-called homological equation\footnote{In fact the homological equation that we will solve contains a small remainder in the right hand side (see \eqref{omoeq}) because we cannot solve all the Fourier modes at the same time.}
\begin{equation}\label{omoeq0}
-\omega\cdot\pa_{\vphi}S+\ii \big[\mathcal{D}^{2}+Z, S\big]+M={\rm Diag}M
\end{equation}
where ${\rm Diag}M$ is defined as in \eqref{newNorm},
then $L^+$ is of the form \eqref{goal} with $Z_+=Z+{\rm Diag}M$ and where
$M_+$ is a sum of terms containing at least two operators of size $\eps$ and thus is formally of size $\eps^2$.

\noindent
Repeating infinitely many times
the same procedure  we will construct a change of variable $\Phi$ such that
$$
\Phi\circ L \circ\Phi^{-1}=L_\infty=\omega\cdot\pa_{\vphi}-\ii \big(\mathcal{D}^{2}+Z_\infty\big)
$$
with $Z_{\infty}$ in normal form according to Definition \ref{normalform200} 
which is our final goal.

\subsection{The homological equation}
\subsubsection{Control of the small divisors}
Let $Z\in \mathcal{M}_{-\beta,s}^{\gamma,\mathcal{O}}$ be in normal form
and denote by $\mu_{k,j}$, $k\in \mathbb{N}$ and $j=1,\ldots, d_{k}$ (see \eqref{dimension}), the eigenvalues 
of the block $Z_{[k]}^{[k]}$.

We define 
the set $\mathcal{O}_{+}\subseteq \mathcal{O}\subseteq\mathcal{G}_0\subset[1/2,3/2]^d$ of parameters $\o$ for which we have a good control of the small divisors. 
Let us fix once for all 
\begin{equation}\label{sceltatau}
\tau>d+2(n-1)\tau_0/\beta+2, 
\end{equation}
with $\tau_0$ in \eqref{zeroMel}.
We set
\begin{equation}\label{calO+}
\begin{aligned} 
\mathcal{O}_{+}\equiv\mathcal{O}_{+}(\g,K)&:=\Big\{\omega\in \mathcal{O}\, : \, 
|\omega\cdot l+\lambda_{k}+\mu_{k,j}
- \lambda_{k'}+\mu_{k',j'}|
\geq \frac{2\gamma}{K^{\tau}}\,,\quad
l\in\mathbb{Z}^{d}\,, \;\; |l|\leq K \\
&\qquad \qquad \quad j=1,\ldots,d_k\,, j'=1,\ldots, d_k'\,, \; k,k'\in\mathbb{N} \,,
\quad (l,k,k')\neq (0,k,k)
\Big\}\,.
\end{aligned}
\end{equation}
We have the following.
\begin{lemma}\label{lem:misuro} Assume that 
$\bral {Z}\brar_{-\beta,s}^{\gamma,\mathcal{O}}\leq \gamma/4$ 
for some $0<\gamma\leq 1/8$ then for 
any $K\geq 1$ we have
\begin{equation}\label{misuro}
{\rm meas}\big(\mathcal{O}\setminus \mathcal{O}_{+}(\g,K)\big)\leq 
C\gamma K^{-\tau+d+2(n-1)\tau_0/\beta+1}
\end{equation}
for some $C=C(s,d,n)>0$. 
\end{lemma}

Before giving the proof of Lemma \ref{lem:misuro}
we recall the following classical result regarding the measure of
sublevels of Lipschitz functions.
\begin{lemma}\label{sublevels}
Let $m\geq1$, $\eta>0$ and let $\mathcal{O}$ be 
 a subset of $\mathbb{R}^{m}$, $m\geq 1$ such that ${\rm meas}(\mathcal{O})<+\infty$.
Consider a Lipschitz function  $f : \mathcal{O}\to \mathbb{R}$ such that
\[
|f|^{lip,\mathcal O}\geq a>0.
\]
Then, setting ${\mathcal{O}}_{\eta}:=\{ x\in \mathcal{O}\; : \; |f(x)|\leq \eta\}$ we have
\[
{\rm meas}\big({\mathcal{O}_{\eta}}\big)\leq 
\frac\eta a{\rm meas}\big({\mathcal{O}}\big)\,.
\]
\end{lemma}

\begin{proof}
Let us set  ${\rm diam}({\mathcal{O}_{\eta}}):=
{\rm sup}_{x_1,x_2\in {\mathcal{O}}}|x_1-x_2|$. Notice that
${\rm meas}({\mathcal{O}_{\eta}})\leq {\rm diam}({\mathcal{O}_{\eta}})$.
For any $x_1,x_2\in {\mathcal{O}_{\eta}}$ such that $x_1\neq x_2$, 
we have that
\[
f(x_1)-f(x_2)=\left(\frac{f(x_1)-f(x_2)}{x_1-x_2}\right)(x_1-x_2)\qquad
\Rightarrow\qquad
|x_1-x_2|\leq \frac{{\rm sup}_{x\in {\mathcal{O}_{\eta}}}|f(x)|}{a}\,.
\]
This  implies the thesis.
\end{proof}

\begin{proof}[Proof of Lemma \ref{lem:misuro}]
We write
\begin{equation*}
\mathcal{O}\setminus \mathcal{O}_{+}=\bigcup_{\substack{l\in \mathbb{Z}^{d}
, |l|\leq K \\ k,k'\in \mathbb{N}\\ (\ell,k,k')\neq(0,k,k)}}\bigcup_{\substack{j=1,\ldots,d_{k}\\
j'=1,\ldots,d_{k'}}}R_{l,k,k'}^{j,j'}
\end{equation*}
where
\begin{equation*}
R_{l,k,k'}^{j,j'}:=\Big\{\omega\in \mathcal{O}\, : \, 
|\omega\cdot l+\lambda_{k}+\mu_{k,j}
- \lambda_{k'}+\mu_{k',j'}|\leq\frac{2\gamma}{K^{\tau}}
\Big\}.
\end{equation*}
We claim that, for $k\neq k'$, $l\neq0$,
\begin{equation}\label{badclaim}
{\rm if}  \qquad R_{l,k,k'}^{j,j'}\neq  \emptyset \qquad {\rm then}\qquad   k+k'\leq C|l|
\end{equation}
for some constant $C>0$ depending only on $n,d$ and $|\omega|$.
Indeed, by hypothesis, there is $\omega\in \mathcal{O}$
such that
\begin{equation}\label{lava}
|\lambda_{k}+\mu_{k,j}
- \lambda_{k'}+\mu_{k',j'}|\leq \frac{2\gamma}{K^{\tau}}+
|\omega\cdot l|\leq C|l|+\frac14\,.
\end{equation}
On the other hand, since $\bral {Z}\brar_{-\beta,s}^{\gamma,\mathcal{O}}\leq \gamma/4$ ,
by  Lemma \ref{herm} and Corollary \ref{coromu}, we have that 
 \begin{equation}\label{condEigen}
 |\mu_{k,j}|^{sup,\mathcal{O}}\leq \frac{\gamma}{4|k|^{\beta}}\,, \qquad 
 |\mu_{k,j}|^{lip,\mathcal O}\leq \frac{1}{4|k|^{\beta}}\,.
 \end{equation}
Then   using \eqref{eigen} and the first in \eqref{condEigen}, 
we conclude for $k\neq k'$
\begin{equation}\label{lava2}
|\lambda_{k}+\mu_{k,j}
- \lambda_{k'}+\mu_{k',j'}|\geq 
 \frac{1}{2}(k+k')\,.
\end{equation}
Hence, by \eqref{lava}, we have
\[
C|l|\geq \frac{1}{2}(k+k')-\frac14\geq \frac{1}{4}(k+k')
\]
 which implies \eqref{badclaim}.\\
 We also notice that when $l=0$ and $k\neq k'$ then $R_{l,k,k'}^{j,j'}=  \emptyset$ for all $j,j'$. Indeed in  such case, using again \eqref{lava2}, 
 we get $|\omega\cdot l+\lambda_{k}+\mu_{k,j}
- \lambda_{k'}+\mu_{k',j'}|
\geq \frac{1}{2}|k+k'|\geq  \frac{1}{2}>\frac{2\gamma}{K^{\tau}}$.

\noindent
Let us now consider the case $l\neq0$ and $k=k'$. We claim that
\begin{equation}\label{claim222}
|k|\geq  |l|^{\frac{\tau_0}{\beta}}\qquad \Rightarrow 
\qquad R_{l,k,k}^{j,j'}=\emptyset\,.
\end{equation}
We recall that, by assumption, the set $\mathcal{O}$ is contained in the set 
$\mathcal{G}_0$ in \eqref{zeroMel}.
Hence, for $\omega\in {\mathcal{O}}$, we deduce  by \eqref{condEigen}
\[
|\omega\cdot l+\mu_{k,j}-\mu_{k,j'}|\geq |\omega\cdot l|
-\Big(|\mu_{k,j}|^{sup,\mathcal{O}}+|\mu_{k,j'}|^{sup,\mathcal{O}}\Big)\geq
\frac{4\gamma}{|l|^{\tau_0}}-\frac{\gamma}{2|k|^{\beta}}\geq\frac{2\gamma}{|l|^{\tau_0}}
\]
using that $|k|^{\beta}\geq |l|^{\tau_0}$ 
which implies claim \eqref{claim222} since $\tau_0<\tau$.

\noindent
Now it remains to estimate the measure of
\begin{equation*}
\bigcup_{\substack{l\in \mathbb{Z}^{d}
, 0<|l|\leq K \\ |k|,|k'|\leq CK}}\bigcup_{\substack{j=1,\ldots,d_{k}\\
j'=1,\ldots,d_{k'}}}R_{l,k,k'}^{j,j'}
\end{equation*}
In order to estimate the measure of a single \emph{bad} set 
$R_{l,k,k'}^{j,j'}$ we compute the Lipschitz norm of the function
$$f(\o)=\omega\cdot l+\lambda_{k}+\mu_{k,j}(\o)
- \lambda_{k'}+\mu_{k',j'}(\o)\,.$$ 
 The second condition in \eqref{condEigen} 
 implies that (recall that $l\neq0$)
 \[
 |f|^{lip,\mathcal O}\geq \frac12\,.
 \]
Then
Lemma \ref{sublevels} implies that
${\rm meas}(R_{l,k,k'}^{j,j'}) \leq 2\frac{\g}{K^\tau}$.
Finally, we recall that, by \eqref{dimension}, \eqref{badclaim} and \eqref{claim222},
we have that 
\[
d_{k}d_{k'}\leq |l|^{2(n-1)} \quad {\rm if} \quad k\neq k' \quad 
{\rm and} \quad d_{k}^{2}\leq |l|^{\frac{2(n-1)\tau_0}{\beta}}\quad {\rm if} \quad k= k'\,.
\]
Hence
\begin{equation*}
\begin{aligned}
{\rm meas}\big(\mathcal{O}\setminus\mathcal{O}_{+}\big)&\leq
\sum_{\substack{l\in \mathbb{Z}^{d}
, 0<|l|\leq K \\ |k|,|k'|\leq CK}}\sum_{\substack{j=1,\ldots,d_{k}\\
j'=1,\ldots,d_{k'}}}R_{l,k,k'}^{j,j'}
{\leq }
\sum_{\substack{l\in \mathbb{Z}^{d}
, 0<|l|\leq K \\ |k|,|k'|\leq CK}}2\frac{\g}{K^\tau} 
d_{k}d_{k'}
\leq    
C\g K^{d+\frac{2(n-1)\tau_0}{\beta}+1-\tau}\,,
\end{aligned}
\end{equation*}
which is the \eqref{misuro}.
\end{proof}

\subsubsection{Resolution of the Homological equation}
In this section we solve the following homological equation equation 
\begin{equation}\label{omoeq}
-\omega\cdot\pa_{\vphi}S+\ii \big[\mathcal{D}^{2}+Z, S\big]+M={\rm Diag}M+R
\end{equation}
where ${\rm Diag}M$ is defined as in \eqref{newNorm}
and $R$ is some remainder to be determined.
 
 \begin{lemma}{\bf (Homological equation)}\label{omoequation} Let 
 $Z\in \mathcal{M}^{\gamma,\mathcal{O}}_{-\beta,s}$ in normal form
 and  $M\in \mathcal{M}^{\gamma,\mathcal{O}}_{-\beta,s,\s}$.
  Assume that $\bral{Z}\brar_{-\beta,s}^{\gamma,\mathcal{O}}\leq \gamma/4$ 
  and let $0<\s_+<\s$ such that 
 \begin{equation}\label{constraintS}
 \s-\s_+\geq K^{-1}\,.
 \end{equation}
 For any $\omega\in \mathcal{O}_+\equiv  \mathcal{O}_+(\gamma,K) $ 
 (defined in \eqref{calO+}) 
 there exist Hamiltonian operators 
 $S,R\in \mathcal{M}^{\gamma,\mathcal{O_+}}_{-\beta,s,\s_+}$ 
 satisfying
 \begin{align}
 &\bral S\brar_{-\beta,s,\s_+}^{\gamma,\mathcal{O}_+}\leq_{s}  
\frac{K^{2\tau+\frac{n}{\beta}\tau+n+2d+1}}{\gamma}
\bral M \brar_{-\beta,s,\s}^{\gamma,\mathcal{O}} 
 \label{gene}\\
&\bral R\brar_{-\beta,s,\s_+}^{\gamma,\mathcal{O}_+}\leq_{s}  
\bral M \brar_{-\beta,s,\s}^{\,\mathcal{O}} 
K^{d}e^{-(\s-\s_+)K}
\label{ultrastima}
 \end{align}
such that equation \eqref{omoeq} is satisfied.
 \end{lemma}
 
 \begin{proof}
  The proof  is an adaptation of (for instance) Lemma $4.3$
 in \cite{GP}.
We set 
\begin{equation}\label{restoUltra}
R_{[k]}^{[k']}(l)=M_{[k]}^{[k']}(l)\,,\quad k,k'\in \mathbb{N}\,,
\quad l\in\mathbb{Z}^{d}\,,\quad |l|>K
\end{equation}
and $R_{[k]}^{[k']}(l)=0$ for $|l|\leq K$. 
By Lemma \ref{DecayAlg2} and \eqref{constraintS} one deduces the \eqref{ultrastima}.
Moreover, recalling \eqref{newNorm},
we have that
 equation \eqref{omoeq} is equivalent to 
\begin{equation}\label{omoeq2}
\mathcal{G}(l,k,k',\omega)S_{[k]}^{[k']}(l)+
M_{[k]}^{[k']}(l)=0
\end{equation}
for any $l\in \mathbb{Z}^{d}$, $k,k'\in \mathbb{N}$ with $(l,k,k')\neq (0,k,k)$
where the operator $\mathcal{G}(l,k,k',\omega)$ is the linear operator
acting on complex $d_{k}\times d_{k'}$-matrices as
\begin{equation}\label{azioneG}
\mathcal{G}(l,k,k',\omega)A:=\Big[-\ii \omega\cdot l +\ii \big( \lambda_{k}{\rm Id}_{[k]}+Z_{[k]}^{[k]}\big)\Big]
A-
\ii A
\big( \lambda_{k'}{\rm Id}_{[k']}+Z_{[k']}^{[k']}\big)\,.
\end{equation}
Now, since $Z_{[k]}^{[k]}$ is Hermitian, there is 
a orthogonal $d_{k}\times d_{k}$-matrix $U_{[k]}$
such that
\[
U_{[k]}^{T} \big( \lambda_{k}{\rm Id}_{[k]}+Z_{[k]}^{[k]}\big) U_{[k]}
=D_{[k]}:={\rm diag}_{j=1,\ldots,d_{k}}\big(\lambda_{k}+\mu_{k,j}\big)\,,
\]
where $\mu_{k,j}$ 
are the eigenvalues of $Z_{[k]}^{[k]}$.
By setting
\begin{equation*}
\widehat{S}_{[k]}^{[k']}(l):=U_{[k]}^{T} S_{[k]}^{[k']}(l)U_{[k']}\,,
\qquad 
\widehat{M}_{[k]}^{[k']}(l):=U_{[k]}^{T} M_{[k]}^{[k']}(l)U_{[k']}
\end{equation*}
equation \eqref{omoeq2}
reads
\begin{equation}\label{omoeq3}
\Big(-\ii \omega\cdot l +\ii D_{[k]}\Big)\widehat{S}_{[k]}^{[k']}(l)-\ii
\widehat{S}_{[k]}^{[k']}(l) D_{[k']}+\widehat{M}_{[k]}^{[k']}(l)=0\,.
\end{equation}
For $\omega\in \mathcal{O}_{+}$ (see \eqref{calO+})
the solution of \eqref{omoeq3} is
 given by (recalling the notation \eqref{notazioneBlock})
\begin{equation}\label{solOmoeq}
\widehat{S}_{k,j}^{k',j'}(l):=\left\{\begin{aligned}
& 0\,, \qquad \qquad\qquad |l|>K\,\;\; {\rm or}\;\; l=0\; {\rm and } \; k=k'\,,\\
& \frac{\ii \widehat{M}_{k,j}^{k',j'}(l)}{-\omega\cdot l
+\lambda_{k}+\mu_{k,j}-\lambda_{k'}-\mu_{k',j,}}\,,
\qquad {\rm otherwise}\,.
\end{aligned}\right.
\end{equation}
Since $M$ is Hamiltonian (see Def. \ref{Hermo} and \eqref{albero3}) 
it is easy to check that
also $S$ is Hamiltonian.
We claim that
\begin{equation}\label{claimo}
\|{S}_{[k]}^{[k']}(l)\|_{\mathcal{L}(L^2)}=\|\widehat{S}_{[k]}^{[k']}(l)\|_{\mathcal{L}(L^2)}\leq_{s}
\frac{ K^{\tau+\frac{n }{2\beta}\tau+\frac n2}}{\gamma}\|\widehat{M}_{[k]}^{[k']}(l)\|_{\mathcal{L}(L^2)}=
\frac{ K^{\tau+\frac{n }{2\beta}\tau+\frac n2}}{\gamma}
\|{M}_{[k]}^{[k']}(l)\|_{\mathcal{L}(L^2)}\,.
\end{equation}

\begin{proof}[Proof of the claim \eqref{claimo}]
To prove the claim we follows the strategy used in the proof of Lemma 4.3 in \cite{GP} (see also   Proposition 2.2.4 in \cite{DS1}) and
we prove \eqref{claimo} considering three different regimes 
of the indexes $k,k'$.

\smallskip
\noindent
{\bf Case 1.} Assume that
\begin{equation}\label{caso1}
{\rm max}\{ k, k'\}> K_1 {\rm min}\{k, k'\}
\end{equation}
for some $K_1>0$ large to be determined.
Without loss of generality we can assume ${k}>K_1{k'}$.
We note that 
\begin{equation}\label{caso1bis}
|-\omega\cdot l
+\lambda_{k}+\mu_{k,j}|\geq\frac{1}{4} \lambda_{k}
\end{equation}
 if $\lambda_{k}\geq K_1^{2}
 \geq |\omega | K\geq |\omega\cdot l|$ and using that, by hypothesis on $Z$,  $|\mu_{k,j}|\leq 1/4$. We choose 
 $K_1:=8 K$.
Equation \eqref{omoeq3} can be written
\[
({\rm Id}+\mathcal{B}_{k,k'}(l))\widehat{S}_{[k]}^{[k']}(l) 
+\big(-\ii \omega\cdot l +\ii D_{[k]}\big)^{-1}\widehat{M}_{[k]}^{[k']}(l) =0
\]
where $\mathcal{B}_{k,k'}(l)\widehat{S}_{[k]}^{[k']}(l):=
\big(-\ii \omega\cdot l +\ii D_{[k]}\big)^{-1}\widehat{S}_{[k]}^{[k']}(l)
\ii D_{[k']}$.
Since
\[
\|\mathcal{B}_{k,k'}(l)\widehat{S}_{[k]}^{[k']}(l)\|_{\mathcal{L}(L^2)}
\stackrel{\eqref{caso1bis}}{\leq}\frac{2\lambda_{k'}}{\lambda_{k}}
\|\widehat{S}_{[k]}^{[k']}(l)\|_{\mathcal{L}(L^2)}
\stackrel{\eqref{caso1}}{\leq}\frac{1}{2}
\|\widehat{S}_{[k]}^{[k']}(l)\|_{\mathcal{L}(L^2)}\,,
\]
thanks to the fact that $K_1\geq 8$, we have that the operator  
$({\rm Id}+\mathcal{B}_{k,k'}(l))$ is invertible
using Neumann series. Therefore   we have
\begin{equation}\label{1cas}
\|\widehat{S}_{[k]}^{[k']}(l)\|_{\mathcal{L}(L^2)}
\leq_{s}\|\widehat{M}_{[k]}^{[k']}(l)\|_{\mathcal{L}(L^2)}\,.
\end{equation}

\smallskip
\noindent
{\bf Case 2.} Assume that
\begin{equation}\label{caso2}
{\rm max}\{ {k}, {k'}\}\leq  
K_1 {\rm min}\{{k},{k'}\}\,,\qquad {\rm and} \qquad
{\rm max}\{ {k}, {k'}\}> K_2\,,
\end{equation}
for some $K_2>0$ to be determined.
The \eqref{caso2} implies that
\begin{equation}\label{caso2bis}
{\rm min}\{{k},{k'}\}\geq K_2 K_1^{-1}\,.
\end{equation}
Using Corollary \ref{coromu} we also note that for all $k$
\begin{equation}\label{estimmu}
|\mu_{[k]}|\leq  \frac\gamma{
4\langle k\rangle^{\beta}}\,.
\end{equation}
and thus
\begin{equation}\label{caso2tris}
\|D_{[k]}-\lambda_{k}{\rm Id}_{[k]}\|_{\mathcal{L}(L^2)}\leq 
\frac\gamma{
4\langle k\rangle^{\beta}}\,.
\end{equation}
Equation \eqref{omoeq3} is equivalent to
\begin{equation}\label{scelta34}
\big({\rm Id}+\mathcal{B}_{k,k'}^{+}(l)\big)\widehat{S}_{[k]}^{[k']}(l)
+
\frac{1}{-\omega\cdot l+\lambda_{k}-\lambda_{k'}}
\widehat{M}_{[k]}^{[k']}(l)=0\,,
\end{equation}
where the operator $\mathcal{B}_{k,k'}^{+}(l)$ acts on
$d_{k}\times d_{k'}$-matrices as
\[
\mathcal{B}_{k,k'}^{+}(l)\widehat{S}_{[k]}^{[k']}(l)=
\frac{1}{-\omega\cdot l+\lambda_{k}-\lambda_{k'}}
\Big[
\big(D_{[k]}-\lambda_{k}{\rm Id}_{[k]}\big)\widehat{S}_{[k]}^{[k']}(l)
- \widehat{S}_{[k]}^{[k']}(l)
\big(D_{[k']}-\lambda_{k'}{\rm Id}_{[k']}\big)
\Big]\,.
\]
We need to estimate  the operator norm of  $\mathcal{B}_{k,k'}^{+}(l)$.
First notice that, for any $\omega\in \mathcal{O}_{+}$ (see
\ref{calO+}),
\begin{equation}\label{smallNormale}
\begin{aligned}
|-\omega\cdot l&+\lambda_{k}-\lambda_{k'}|
\geq
|\omega\cdot l+\lambda_{k}+\mu_{k,j}
- \lambda_{k'}+\mu_{k',j'}|-
\big(|\mu_{[k']}|^{sup,\mathcal{O_+}}+|\mu_{[k']}|^{sup,\mathcal{O_+}}\big)\\
&\stackrel{\eqref{estimmu}}{\geq} \frac{2\gamma}{K^\tau}-\frac\gamma{4\langle k\rangle^{\beta}} -\frac\gamma{4\langle k'\rangle^{\beta}}
\stackrel{\eqref{caso2bis}}{\geq} \frac{2\gamma}{K^\tau}-
\frac{\gamma}{ 2K_2K_1^{-1}}
\geq\frac{\gamma}{K^\tau} 
\end{aligned}
\end{equation}
providing 
\begin{equation}\label{KK}(K_2K_1^{-1})^\beta\geq K^\tau\,.\end{equation}
Combining
 \eqref{caso2tris} and \eqref{smallNormale} we get
that, in operator norm, 
\begin{equation}\label{scelta33}
\|\mathcal{B}_{k,k'}^{+}(l)\|_{\mathcal{L}(L^2)}\leq
\frac{K^\tau}{4}\big(\langle k\rangle^{-\beta}
+\langle k'\rangle^{-\beta}\big)
\leq \frac12
\end{equation}
providing  \eqref{KK}.
 Recalling $K_1= 8 K$
we choose 
\begin{equation}\label{sceltaK2}
K_{2}:=8K^{\frac{\tau}{\beta}+1}\,.
\end{equation}
Now, by \eqref{scelta33}, the operator 
$\big({\rm Id}+\mathcal{B}_{k,k'}^{+}(l)\big)$ is invertible, and hence
by \eqref{scelta34} and \eqref{smallNormale} we get
\begin{equation}\label{2cas}
\|\widehat{S}_{[k]}^{[k']}(l)\|_{\mathcal{L}(L^{2})}\leq
\gamma^{-1} 2K^{\tau}\|\widehat{M}_{[k]}^{[k']}(l)\|_{\mathcal{L}(L^2)}\,.
\end{equation}

\smallskip
\noindent
{\bf Case 3.} Assume that
\begin{equation}\label{caso3}
{\rm max}\{ {k}, {k'}\}\leq  
K_1 {\rm min}\{{k},{k'}\}\,,\qquad {\rm and} \qquad
{\rm max}\{ {k}, {k'}\}\leq K_2\,,
\end{equation}
In that case the size of the blocks are less than $K_2^{n}$ 
and we have, for any 
$j=1,\ldots, d_{k}$, $j'=1,\ldots, d_{k'}$,
\[
|\widehat{S}_{k,j}^{k',j'}(l)|\leq \gamma^{-1}K^{\tau}
|\widehat{M}_{k,j}^{k',j'}(l)|\,,
\] 
and hence
\begin{equation}\label{3cas}
\|\widehat{S}_{[k]}^{[k']}(l)\|_{\mathcal{L}(L^{2})}
\leq \gamma^{-1} K^{\tau}K_{2}^{\frac{n}{2}}
\|\widehat{M}_{[k]}^{[k']}(l)\|_{\mathcal{L}(L^{2})}
\stackrel{\eqref{sceltaK2}}{\leq_s}
\gamma^{-1} K^{\tau+\frac{n }{2\beta}\tau+\frac n2}
\|\widehat{M}_{[k]}^{[k']}(l)\|_{\mathcal{L}(L^{2})}\,.
\end{equation}
By collecting the bounds \eqref{1cas}, \eqref{2cas}
and \eqref{3cas} we get \eqref{claimo}.
\end{proof}

\noindent
Estimate \eqref{claimo}
allows us to conclude that
\begin{equation}\label{claimo2}
\bral \widehat{S}\brar_{s,\s_+}=\bral {S}\brar_{s,\s_+}
\leq_s
 \frac{K^{\tau+\frac{n }{2\beta}\tau+\frac n2}}{
 \gamma (\s-\s_+)^{d}}\bral M \brar_{s,\s}\stackrel{\eqref{constraintS}}{\leq}
\frac{K^{\tau+\frac{n}{2\beta}\tau+\frac n2+d}}{\gamma}
 \bral M \brar_{s,\s}\,.
\end{equation}
Indeed (recall \eqref{decayNorm2})
\[
\begin{aligned}
|\mathcal{D}S|_{s,\s'}^{2}&\stackrel{\eqref{claimo}}{\leq_{s}}
\gamma^{-2}
\sum_{l\in \mathbb{Z}^{d},h\in \mathbb{N}}\langle l,h\rangle^{2s}
e^{2|l|\s'}\sup_{|k-k'|=h} K^{2\tau+\frac{n }{\beta}\tau+{n}+2d}
\|(\mathcal{D}M)_{[k]}^{[k']}(l)\|^{2}_{\mathcal{L}(L^{2})}\,.
\end{aligned}
\]
To obtain \eqref{gene}, it remains to estimate the Lipschitz variation of the matrix $S$.
For any family of operators $\omega\mapsto A=A(\omega)$ and 
 any  $\omega_1,\omega_2\in \mathcal{O}_{+}$ with 
 $\omega_1\neq\omega_2$ we set
\[
\Delta_{\omega_1,\omega_2}A:=\frac{A({\omega_1})-A(\omega_2)}{\omega_1-\omega_2}\,.
\]
Hence, by \eqref{omoeq2}, we obtain
\begin{equation}\label{omoeq4}
\mathcal{G}(l,k,k',\omega_2)\Delta_{\omega_1,\omega_2}S_{[k]}^{[k']}(l)
+\Delta_{\omega_1,\omega_2}M_{[k]}^{[k']}(l)+
\Delta_{\omega_1,\omega_2} \mathcal{G}(l,k,k',\cdot) 
S_{[k]}^{[k']}(l,\omega_1)=0\,,
\end{equation}
which is an equation of the same form of \eqref{omoeq2}
with different non-homogeneous term.
Using that $|Z|^{lip,\mathcal O}\leq1/4$ we deduce from \eqref{azioneG} 
\[
\|\Delta_{\omega_1,\omega_2}\mathcal{G}(l,k,k',\cdot) 
S_{[k]}^{[k']}(l)
 \|_{\mathcal{L}(L^{2})}
 \leq_{s} K\|S_{[k]}^{[k']}(l) \|_{\mathcal{L}(L^{2})}
 \stackrel{\eqref{claimo}}{\leq_s}
 \gamma^{-1} K^{\tau+\frac{n }{2\beta}\tau+\frac n2+1}
\|\widehat{M}_{[k]}^{[k']}(l)\|_{\mathcal{L}(L^2)}\,.
\]
Then, reasoning as in the proof of \eqref{claimo}, we deduce
\[
\| \Delta_{\omega_1,\omega_2}S_{[k]}^{[k']}(l)\|_{\mathcal{L}(L^{2})}
\leq_{s}
\frac{K^{2\tau+\frac{n }{\beta}\tau+n+2}}{\gamma^{2}}
\|M_{[k]}^{[k']}(l)\|_{\mathcal{L}(L^{2})}
+
\frac{K^{\tau+\frac{n }{2\beta}\tau+\frac{n}2+1}}{\gamma}
\| \Delta_{\omega_1,\omega_2}M_{[k]}^{[k']}(l)\|_{\mathcal{L}(L^{2})}
\]
which, following the proof of \eqref{claimo2} 
and using \eqref{constraintS} and recalling the choice \eqref{parametrix}, 
implies  \eqref{gene}.
\end{proof}

\subsection{The KAM step}
Now we compute the new $L_+$ (see \eqref{goal}) generated by the change of variable $\Phi=e^S$ where $S$ satisfies the homological equation \eqref{omoeq}. \\
We first prove the following.
\begin{lemma}\label{stimaMappa}
There is $C(s)>0$ (depending only on $s$) such that, if
\begin{equation}\label{smallsmall}
\gamma^{-1} C(s)K^{2\tau+\frac{n}{\beta}\tau+n+2d+1}
\bral M\brar_{-\beta,s,\s}^{\gamma,\mathcal{O}}\leq \frac{1}{2}\,,
\end{equation}
then  the map 
$\Phi=e^{S}={\rm Id}+\Psi$, with $S$ given by Lemma \ref{omoequation},
satisfies 
\begin{equation}\label{stimaMappa1}
\bral \Psi\brar_{-\beta,s,\s_+}^{\gamma,\mathcal{O}_+}\leq_s 
\gamma^{-1} K^{2\tau+\frac{n}{\beta}\tau+n+2d+1}
\bral M\brar_{-\beta,s,\s}^{\gamma,\mathcal{O}}\,.
\end{equation}
\end{lemma}

\begin{proof}
By \eqref{gene} and \eqref{smallsmall}
we have that
\begin{equation}\label{vulcano}
C(s)\bral S\brar_{-\beta,s,\s_+}^{\gamma,\mathcal{O}_+}
\leq_{s}
1/2\,.
\end{equation}
This implies te smallness condition \eqref{mito30}. Hence 
the \eqref{stimaMappa1} follows by Lemma \ref{well-well}.
\end{proof}

\subsubsection{The new normal form}
As said in section \ref{KAMstep} we define the new normal form $Z_{+}$ 
as
\begin{equation}\label{newZ}
Z_{+}:=Z+\ii {\rm Diag}M\,.
\end{equation}
We have the following.

\begin{lemma}{\bf (New normal form)}\label{newnormal}
We have that $Z_{+}$ in \eqref{newZ} is in normal form
(see Def. \ref{normalform200})
 and satisfies 
 \begin{equation}\label{parole}
 \bral Z_{+}\brar_{-\beta,s}^{\gamma,\mathcal{O}_{+}}
 \leq 
 \gamma(\Theta+\e)\,.
 \end{equation}
 There is a sequence of Lipschitz function
 \[
 \mu_{[k]}^{+} : \mathcal{O}_{0} \to \mathbb{R}^{d_{k}}\,,\quad k\in \mathbb{N}
 \]
such that, for $\omega\in \mathcal{O}_{+}$, the functions
 $\mu_{k,j}^{+}$, for $j=1,\ldots,d_{k}$, are the eigenvalues
 of the block $(Z_{+})_{[k]}^{[k]}$ 
satisfying 
\begin{equation}\label{modo1}
\sup_{k\in \mathbb{N}}\langle k\rangle^{\beta}
 |\mu^{+}_{[k]} |^{\gamma,\mathcal{O}_0}\leq \gamma
 (\Theta+\e)\,.
\end{equation}
\end{lemma}

\begin{proof}
The matrix $Z_{+}$ is $\vphi$-independent, block-diagonal and Hermitian 
by construction. Estimate \eqref{modo1} is a consequence of Corollary \ref{coromu}.
\end{proof}

\subsubsection{The new remainder}
Now we compute and estimate $M_+$ given by \eqref{goal}.
\begin{lemma}{\bf (The new remainder)}\label{nuovoresto}
Assume that the smallness condition \eqref{smallsmall}
holds true.
 The new remainder 
$M_{+}\in \mathcal{M}^{\gamma,\delta,\mathcal{O}_{+}}_{-\beta,s,\s_+}$
is Hamiltonian and satisfies
\begin{equation}\label{nuovoRem1}
\bral M_{+}\brar_{-\beta,s,\s_+}^{\gamma,\mathcal{O}_+}\leq_s 
K^{2\tau+\frac{n}{\beta}\tau+n+2d+1}
\bral M\brar_{-\beta,s,\s_+}^{\gamma,\mathcal{O}}\Big(e^{-(\s-\s_+)K}
+\gamma^{-1}\bral M\brar_{-\beta,s,\s_+}^{\gamma,\mathcal{O}}\Big)\,.
\end{equation}
\end{lemma}

\begin{proof}
Equations \eqref{omoeq}, \eqref{newZ} and \eqref{L+} lead to the following formula for $M_+$
\begin{equation}\label{gene3}
M_+:=R+\widetilde{M}_{+}:=
R+
\sum_{p\geq 2}\frac{1}{p!}{\rm ad}_{S}^{p-1}\Big(
{\rm Diag}M+R
\Big) +\sum_{p\geq1}\frac{1}{p!}{\rm ad}_{S}^{p}\big(M\big)
\end{equation}
with $R$ satisfying  \eqref{ultrastima}. Thus, in order to prove \eqref{nuovoRem1} we need to estimate $\widetilde{M}_+$.
By \eqref{ultrastima} and \eqref{stim1}, we have
\[
\bral
\big[S, {\rm Diag}M+R\big]
\brar_{-\beta,s,\s_{+}}^{\gamma,\mathcal{O}_+}
\leq C(s)K^{d}
\bral S\brar_{-\beta,s,\s_+}^{\gamma,\mathcal{O}_+}
\bral M\brar_{-\beta,s,\s}^{\gamma,\mathcal{O}}
\]
for some $C(s)>0$. The term $[S,M]$ can be estimated in the same way.
Hence
\begin{equation}\label{gene2}
\begin{aligned}
\bral M_+\brar_{-\beta,s,\s_+}^{\gamma,\mathcal{O}_+}&\leq  K^{d}C(s)
\bral S\brar_{-\beta,s,\s_+}^{\gamma,\mathcal{O}_+}
\bral M\brar_{-\beta,s,\s}^{\gamma,\mathcal{O}}
\sum_{p\geq2}\frac{1}{p!}(C(s))^{p-1}
\big(\bral S\brar_{-\beta,s,\s_+}^{\gamma,\mathcal{O}_+}\big)^{p}\\
&\stackrel{\eqref{vulcano}}{\leq} 4C(s) K^{d}
\bral S\brar_{s,\s_+}^{\gamma,\mathcal{O}_+}
\bral M\brar_{s,\s}^{\gamma,\mathcal{O}}\,.
\end{aligned}
\end{equation}
Using formula \eqref{gene3}
we have that the estimates 
\eqref{ultrastima}, \eqref{gene2} and 
\eqref{gene} imply the \eqref{nuovoRem1}.
By Lemma \ref{conjHam} we have that $M_{+}$ is Hamiltonian.
\end{proof}

\subsection{The iterative Lemma}
We fix $1<\chi<2$, 
$K_0\geq1$, $\s_0:=\s/2$ (see \eqref{conj2}) and we recall that $\Theta_0=\eps_0>0$ (see \eqref{small}).
For $k\in\mathbb{N}$ we introduce the following parameters: 
\begin{equation}\label{parametri}
K_{k}:= 4^{k} {K_0}\,,\quad \s_{k+1}:=(1-2^{-k-3})\s_{k}\,,
\quad
\Theta_{k}:=\Theta_{0}\Big(1+\sum_{0<\nu\leq k }2^{-k}\Big)\,,\quad \eps_k=\eps_0e^{-\chi^k}\,.
\end{equation}
Consider an operator $L_0$ of the form \eqref{start} with $\mathcal{O}
\rightsquigarrow\mathcal{G}_0\cap\mathcal{O}$, $\s\rightsquigarrow \s_0$
where 
 $\mathcal{G}_0$ is in 
\eqref{zeroMel}.
We prove the following.
\begin{proposition}{\bf (Iterative Lemma)}\label{lemmaitero}
There are  $K_{\star}, \Theta_{\star}>0$ depending on $n,s,d$, $\chi$,
with $1<\chi<2$,
such that
if 
\[
\eps_0=\Theta_0\leq \Theta_{\star}\quad \text{and}\quad K_0\geq K_\star
\]
then for all $k\geq 0$ we can construct:

\vspace{0.5em}
\noindent
$\bullet$ sets $\mathcal{O}_{k+1}\subset \mathcal{O}_{k}\subset \mathcal{G}_0$ satisfying
\begin{equation}\label{measO}
{\rm meas}\Big(\mathcal{O}_{k+1}\setminus \mathcal{O}_{k}\Big)
\leq C(s)\gamma K_k^{-\tau+d+2(n-1)\tau_0/\beta+1}
\end{equation}

\noindent
$\bullet$ Lipschitz family of canonical change of variables $\Phi_{k}\equiv\Phi_{k}(\omega):={\rm Id}+\Psi_{k}$
with $\Psi_{k}\in \mathcal{M}_{-\beta,s,\s_{k}}^{\gamma,\mathcal{O}_{k}}$
and 
\begin{equation}\label{stimaPsinu}
\bral \Psi_{k}\brar_{-\beta,s,\s_{k}}^{\gamma,\mathcal{O}_{k}}\leq \e_0 C_{\star}
2^{-k}\,
\end{equation}
where $C_\star=C_\star(n,d,s,K_0)$.

\vspace{0.5em}
\noindent
$\bullet$ Lipschitz family of operators
\begin{equation}\label{opLnu}
L_{k}\equiv L_k(\o):=\omega\cdot\pa_{\vphi}-\ii \big(\mathcal{D}^{2}+Z_{k}\big)+
M_{k}
\end{equation}
with $Z_{k}\in\mathcal{M}_{-\beta,s}^{\gamma,\mathcal{O}_{k}}$ 
in normal form and 
$M_{k}\in\mathcal{M}_{-\beta,s,\s_{k}}^{\gamma,\mathcal{O}_{k}}$ 
Hamiltonian (see Def.  \ref{normalform200}, \ref{Hermo}) 
satisfying 
\begin{equation}\label{Zk} 
Z_{k+1} =Z_k+{\rm Diag}M_k\,,  
\end{equation}
\begin{equation}\label{ittero}
\gamma^{-1}
\bral M_{k}\brar_{-\beta,s,\s_{k}}^{\gamma,\mathcal{O}_{k}}\leq \eps_k\,,
\qquad
\gamma^{-1}
\bral Z_{k}\brar_{-\beta,s}^{\gamma,\mathcal{O}_{k}}\leq \Theta_k\,,
\end{equation}
such that  for any $k\geq1$  
\begin{equation}\label{opLnu2}
L_{k}:=\Phi_{k}\circ L_{k-1}\circ\Phi^{-1}_{k}\quad \forall\, \o\in \mathcal O_k\,.
\end{equation}
\end{proposition}

\begin{proof}[Proof of Proposition \ref{lemmaitero}]
We proceed by induction. At step $k=0$ the operator 
$L_0$ is defined on $\mathcal O_0$ 
by \eqref{start} which is of the form \eqref{opLnu} 
and satisfies \eqref{ittero}. 
Now assume  that we have construct the 
sets $\mathcal O_p$, the operators $L_p$ and 
the changes of variables $\Phi_{p}$ for $p=1,\cdots,k$ 
and let us construct them at step $k+1$.

\medskip

Since \eqref{ittero} implies 
$\bral Z_k\brar_{-\beta,s}^{\gamma,\mathcal O_k}\leq \gamma/4$ 
for $\Theta_0$ small 
enough\footnote{Notice that $\s_{k+1}\searrow \s_0/2$. Hence, by  
 \eqref{parametri}, 
 we have
\[
\s_{k}-\s_{k+1}=\s_{k}\frac{1}{2^{k+3}}\geq \s_0\frac{1}{2^{k+4}}
\geq \frac{1}{K_{k}}=\frac{1}{4^{k} {K_0}}
\]
for ${K_0}>0$ large enough.},  we  use Lemmata \ref{lem:misuro} and  \ref{omoequation} to construct $\mathcal O_{k+1}$, $S_{k+1}$ and $R_{k+1}$. 
The set $\mathcal{O}_{k+1}$ is defined as in \eqref{calO+}
with $\mathcal{O}\rightsquigarrow\mathcal{O}_{k}$, $K\rightsquigarrow K_{k}$
and satisfies the \eqref{measO} by Lemma \ref{lem:misuro}.
By the induction hypothesis
\eqref{ittero} we have
\begin{equation}\label{ferrari}
\gamma^{-1} K_{k}^{2\tau+\frac{n}{\beta}\tau+n+2d+1}
\bral M_{k}\brar_{-\beta,s,\s}^{\gamma,\mathcal{O}}{\leq_{s}}\e_0 
K_{k}^{2\tau+\frac{n}{\beta}\tau+n+2d+1}
 e^{-\chi^{k}}\leq \e_0 C_{\star}
2^{-k}\,,
\end{equation}
provided that
\[
C_{\star}\geq \max_{k}\big( {K_0}^{\mathtt{a}}2^{k}4^{k\mathtt{a}}e^{-\chi^{k}}\big)\,,
\qquad \mathtt{a}=2\tau+\frac{n}{\beta}\tau+n+2d+1\,.
\]
The \eqref{ferrari} implies the smallness condition \eqref{smallsmall}
with $K\rightsquigarrow K_{k}$, $M\rightsquigarrow M_{k}$.
Then Lemma
 \ref{stimaMappa}
provides a map $\Phi_{k+1}={\rm Id}+\Psi_{k+1}$
such that
\begin{equation}\label{ferrari2}
\bral \Psi_{k+1}\brar_{-\beta,s,\s_{k+1}}^{\gamma,\mathcal{O}_{k+1}}
\stackrel{\eqref{stimaMappa1}}{\leq_{s}} 
\gamma^{-1} K_{k}^{2\tau+\frac{n}{\beta}\tau+n+2d+1}
\bral M_{k}\brar_{-\beta,s,\s}^{\gamma,\mathcal{O}_{k}}
\end{equation}
which, by \eqref{ferrari},
implies the \eqref{stimaPsinu}.
By Lemmata \ref{newnormal} and  \ref{nuovoresto} we construct 
\[
L_{k+1}:=\Phi_{k+1}\circ L_{k}\circ \Phi_{k+1}^{-1}=
\omega\cdot\pa_{\vphi}-\ii \big(\mathcal{D}^{2}+Z_{k+1}\big)+M_{k+1}
\]
with $M_{k+1}=M_{+}$ Hamiltonian and  
\begin{equation}\label{normalForm+}
Z_{k+1}=Z_+=Z_{k}+{\rm Diag}M_{k}
\end{equation}
is in normal form (see the \eqref{newNorm} in Def. \ref{normalform200}).
Moreover, by the estimate \eqref{parole},
we deduce that
\begin{equation}\label{gold3}
\begin{aligned}
\!\!\gamma^{-1}\bral Z_{k+1}\brar_{-\beta,s}^{\gamma,\mathcal{O}_{k+1}}
&\leq 
(\Theta_{k}+\e_{k})
\stackrel{\eqref{parametri}}{\leq} 
\Theta_{0}
\big(1+\sum_{j=1}^{k}\frac{1}{2^{j}}\big)
+
\e_0 e^{-\chi^{k}}
\leq
\Theta_0\big(1+\sum_{j=1}^{k+1}\frac{1}{2^{j}}\big)=
\Theta_{k+1}\,.
\end{aligned}
\end{equation}
On the other hand we note that Lemma \ref{nuovoresto} implies 
$$
\gamma^{-1}
\bral M_{k+1}\brar_{-\beta,s,\s_{k+1}}^{\gamma,\mathcal{O}_{k+1}}\stackrel{\eqref{nuovoRem1}}{\leq}
K_{k}^{2\tau+\frac{n}{\beta}\tau+n+2d+1}
\e_{k}  \left(
e^{-\sqrt{K_0}}e^{- 2^{k}}+
\e_{k}\right)
$$
where we used that $(\s_k-\s_{k+1})K_k\geq K_0 2^{k-3}\geq 2^k+\sqrt{K_0}$ for $K_0$ large enough.
Hence 
\begin{equation}\label{gold4}
\begin{aligned}
\gamma^{-1}
\bral M_{k+1}\brar_{-\beta,s,\s_{k+1}}^{\gamma,\mathcal{O}_{k+1}}
&\leq 
2K_{k}^{2\tau+\frac{n}{\beta}\tau+n+2d+1}
\e_0(e^{-\sqrt{K_0}}+\e_0)((e^{-\chi^k})^2\\
&\leq 2K_{0}^{2\tau+\frac{n}{\beta}\tau+n+2d+1}
(e^{-\sqrt{K_0}}+\e_0)  4^{k(2\tau+\frac{n}{\beta}\tau+n+2d+1)}
e^{-(2-\chi)\chi^k}\e_0e^{-\chi^{k+1}} \\
&\leq \eps_0 e^{-\chi^{k+1}}:=\eps_{k+1}
\end{aligned}
\end{equation}
provided $\e_0$ small enough and $K_0$ large enough.
The  \eqref{gold3} and \eqref{gold4} yields \eqref{ittero} with  $k\rightsquigarrow k+1$.
\end{proof}

\subsection{Convergence and Proof of Theorem \ref{thm:redu}}
\begin{proof}[{\bf Proof of Theorem \ref{thm:redu}}]
By the smallness condition \eqref{pescerosso} we have that
hypothesis \eqref{tuttopiccolo} holds
for $\epsilon$ sufficiently small. Hence Proposition \ref{riduco}
applies to the operator $G$ in \eqref{start0}.
The operator \eqref{conj} has the form \eqref{start}
with $M_0\rightsquigarrow R_+$, 
$\mathcal{O}\rightsquigarrow\mathcal{G}_0\cap\mathcal{O}$, 
with $\mathcal{G}_0$ in 
\eqref{zeroMel}
and $\s_0:=\s_+=\s/2$
and $\e_0$ in \eqref{small} satisfies $\e_0\leq_s \epsilon$.
So taking again $\epsilon$ small enough we can apply Proposition 
\ref{lemmaitero} for some $K_0\geq K_\star$.\\
Let us define the set
\begin{equation*}
\mathcal{O}_{\infty}:=\cap_{\nu\geq0}\mathcal{O}_{\nu}\,.
\end{equation*}
By \eqref{measO} we deduce
\eqref{measureEst}. We also notice that $\s_k\geq \s_0/2$ for all $k\geq 0$.
Then, by \eqref{Zk} and \eqref{ittero}, 
we deduce that
\[
\bral Z_{k+1}-Z_k\brar_{-\beta,s}^{\gamma,\mathcal{O}_{\infty}}
\leq \eps_k
\]
and thus, since $\sum \eps_k<+\infty$, 
$Z_k$ is a Cauchy sequence in 
$\mathcal{M}^{\gamma,\mathcal{O}_{\infty}}_{-\beta,s}$ 
and we can define the block diagonal hermitian operator
 \[
 \lim_{\nu\to+\infty}Z_{\nu}:=
 Z_\infty\in \mathcal{M}^{\gamma,\mathcal{O}_{\infty}}_{-\beta,s}\,.
 \]
As a consequence of Corollary \ref{coromu} we deduce \eqref{stimaeigen}.\\ 
 Then definig $\widetilde{\Phi}_{\nu}:=\Phi_{1}\circ\Phi_{2}\circ\cdots
\Phi_{k}={\rm Id}+\widetilde{\Psi}_{k}$
we have by \eqref{stimaPsinu}
\[
\widetilde{\Psi}_{k+1}=\widetilde{\Psi}_{k}+(1+\widetilde{\Psi}_{k})\Psi_{k+1}\qquad  \Rightarrow
\qquad
\bral \widetilde{\Psi}_{k+1}-\widetilde{\Psi}_{k}
\brar_{-\beta,s,\s_0/2}^{\gamma,\mathcal{O}_{\infty}}\leq 
C_{\star}\e_02^{-k}\,.
\]
Thus $\widetilde{\Psi}_k$ is a Cauchy sequence in 
$\mathcal{M}^{\gamma,\mathcal{O}_{\infty}}_{-\beta,s,\s_0/2}$ 
and we can define its limit 
$\widetilde{\Psi}_{\infty}\in
\mathcal{M}^{\gamma,\delta,\mathcal{O}_{\infty}}_{-\beta,s,\s_0/2}$ 
which satisfies
\begin{equation}\label{stimeTRA200}
\begin{aligned}
&\bral \widetilde{\Psi}_{\infty}
\brar_{-\beta,s,\s_0/2}^{\gamma,\delta,\mathcal{O}}
\leq C(s)\e_0\leq_s\epsilon \,.
\end{aligned}
\end{equation}  
 Then the map $\Phi_{\infty}:={\rm Id}+\widetilde{\Psi}_{\infty}$
satisfies $\Phi_{\infty}:=\lim_{k\to+\infty}\tilde{\Phi}_{k}$.   
Finally for $\omega\in \mathcal{O}_{\infty}$ we set 
  \begin{equation}\label{mito31}
  \Phi:=\mathcal{T}\circ\Phi_{\infty}:={\rm Id}+\Psi\,,\qquad 
  \Psi:=\mathcal{F}+\widetilde{\Psi}_{\infty}+\mathcal{F}\circ\widetilde{\Psi}_{\infty}
  \end{equation}
  where $\mathcal{T}$ is the map given by Proposition \ref{riduco}.
  Since $\alpha-1\leq -\beta$ (see \eqref{parametrix}), by Remark
  \ref{inclusioni}
   we
  have that $\mathcal{F}$ belongs to 
  $\mathcal{M}_{-\beta,s,\s_0/2}^{\gamma,\mathcal{O}_{\infty}}$.
  The \eqref{stimeTRA}
  follows by composition using \eqref{bracket}, 
  \eqref{stimeTRA200} and \eqref{mito31}.
  The \eqref{stimeTRA1000} follows 
  by Lemma \ref{AzioneDec}.
   The \eqref{coniugato}
   follows by the construction.
\end{proof}

\appendix
\section{Technical Lemmata}\label{techtech}

In this appendix we assume $s> (d+n)/2$ and $\s> 0$.
\begin{lemma}\label{DecayAlg}
Let $A,B\in \mathcal{M}_{s,\s}$. Then
the following holds:
\begin{itemize}

\item[$(i)$] for any $z\in \ell_{s,\s}$ one has 
$\|Az\|_{s,\s}\leq C(s)|A|_{s,\s}\|z\|_{s,\s}$;

\smallskip
\item[$(ii)$] one has $|AB|_{s,\s}\leq C(s)|A|_{s,\s}|B|_{s,\s}$;

\smallskip
\item[$(iii)$] by setting (recall \eqref{timeMat})
$
\Pi_{N}A:=\sum_{|l|<N}A(l)e^{\ii l \cdot \vphi}
$
one has
\[
|({\rm Id}- \Pi_{N})A|_{s,\s'}\leq \frac{C(s)e^{-(\s-\s')N}}{(\s-\s')^{d}}
|A|_{s,\s}\,\quad 0<\s'<\s\,;
\]
\end{itemize}
Similar bounds holds also replacing $|\cdot|_{s,\s}$, $\|\cdot\|_{s,\s}$
with the norms $|\cdot|_{s,\s}^{\gamma,\mathcal{O}}$, 
$\|\cdot\|_{s,\s}^{\gamma,\mathcal{O}}$ respectively.
\begin{itemize}
\item[$(iv)$]
Let $\beta>0$ and 
$A\in \mathcal{M}_{s+\beta,\s}^{\gamma,\mathcal{O}}$
then (recall \eqref{Diag}) 
$\mathcal{D}^{\beta} A \mathcal{D}^{-\beta}, \mathcal{D}^{-\beta} A \mathcal{D}^{\beta}\in 
\mathcal{M}^{\gamma,\mathcal{O}}_{s,\s}$ and 
\begin{equation}\label{beebee}
|\mathcal{D}^{\beta} A \mathcal{D}^{-\beta}|^{\gamma,\mathcal{O}}_{s,\s}
+|\mathcal{D}^{-\beta} A \mathcal{D}^{\beta}|^{\gamma,\mathcal{O}}_{s,\s}
\leq
|A|^{\gamma,\mathcal{O}}_{s+\beta,\s}\,.
\end{equation}

\end{itemize}
\end{lemma}

\begin{proof}
Items $(i)$ and $(ii)$ follows by lemmata $2.6$, $2.7$ in \cite{BCP}.
Item $(iii)$ follows by the definition of the norm in \eqref{decayNorm2}.
To prove item $(iv)$ we reason as follows.
We study the operator 
$\mathcal{D}^{\beta}A \mathcal{D}^{-\beta}$. The bound for 
$\mathcal{D}^{-\beta}A \mathcal{D}^{\beta}$ can be deduced in the same way.
First we note
 that
 \[
 \big(\mathcal{D}^{\beta}A \mathcal{D}^{-\beta}\big)_{[k]}^{[k']}(l)=
 \lambda_{k}^{\frac{1}{2}\beta}
 \lambda_{k'}^{-\frac{1}{2}\beta}A_{[k]}^{[k']}(l)\,.
 \]
 If $k'\geq 1/2 k$ then, recalling \eqref{equiweight},
 we deduce
 \begin{equation}\label{albero12}
  \lambda_{k}^{\frac{1}{2}\beta}
 \lambda_{k'}^{-\frac{1}{2}\beta}\|A_{[k]}^{[k']}(l)\|_{\mathcal{L}(L^2)}\leq_{s}
 \|A_{[k]}^{[k']}(l)\|_{\mathcal{L}(L^2)}\,.
 \end{equation}
If on the contrary $k'\leq 1/2 k$, then $|k-k'|\geq 1/2 k$. Hence
we have
 \begin{equation}\label{albero13}
  \lambda_{k}^{\frac{1}{2}\beta}
 \lambda_{k'}^{-\frac{1}{2}\beta}\|A_{[k]}^{[k']}(l)\|_{\mathcal{L}(L^2)}\leq_{s}
 \|A_{[k]}^{[k']}(l)\|_{\mathcal{L}(L^2)}|k-k'|^{\beta}\,.
 \end{equation}
 Bounds \eqref{albero12} and \eqref{albero13}
imply \eqref{beebee} for the norm $|\cdot|_{s,\s}$. The bound for the 
 Lipschitz norm in \eqref{decayLip}
 follows in the same way.
\end{proof}

\begin{lemma}\label{decayspaziotempo}
Let $A$ be a matrix as in \eqref{timeMat}
with finite $|\cdot|_{s,\s}$ norm.
 Then
\[
|A(\vphi)|_{s}:=
\Big(
\sum_{h\in \mathbb{N}}\langle h\rangle^{2s}
\sup_{|k-k'|=h}\|A_{[k]}^{[k']}(\vphi)\|^{2}_{\mathcal{L}(L^{2})}
\Big)^{\frac{1}{2}}
 \leq_{s}\frac{1}{(\s-\s')^{s_0+d}}
 |A|_{s,\s}\,,\quad \forall\,\vphi\in \mathbb{T}_{\s'}^{d}\,,\,\s'<\s\,.
\]
\end{lemma}
\begin{proof}
For any $\vphi\in \mathbb{T}_{\s'}$ we have (using Cauchy-Schwarz inequality and taking $s_0:=(d+1)/2$)
\[
\begin{aligned}
|A(\vphi)|_{s}^{2}&=\sum_{h\in \mathbb{N}}\langle h\rangle^{2s}
\sup_{|k-k'|=h}\|A_{[k]}^{[k']}(\vphi)\|^{2}_{\mathcal{L}(L^{2})}\leq_{s,s_0}
\sum_{h\in \mathbb{N}}\langle h\rangle^{2s}
\sup_{|k-k'|=h}\sum_{l\in\mathbb{Z}^{d}}
\|A_{[k]}^{[k']}(l)\|^{2}_{\mathcal{L}(L^{2})}e^{2|l|\s'}|l|^{2s_0}\\
&\leq 
\sum_{l\in \mathbb{Z}^{d},h\in \mathbb{N}}\langle l\rangle^{2s}e^{2|l|\s}
\sup_{|k-k'|=h}\|A_{[k]}^{[k']}(l)\|^{2}_{\mathcal{L}(L^{2})} e^{-2(\s-\s')|l|}|l|^{s_0}
\leq_{s}\frac{1}{(\s-\s')^{2(s_0+d)}}|A|^{2}_{s,\s}
\end{aligned}
\]
where we used that the function $e^{-(\s-\s')x}x^{s_0+d}$ has a maximum
in $x=s_0/(\s-\s')$.
\end{proof}

\begin{lemma}\label{DecayAlg2}
Let $\alpha,\beta\in \mathbb{R}$ and 
consider $A\in \mathcal{M}^{\gamma,\mathcal{O}}_{\alpha,s+\beta,\s}$
and
$B\in \mathcal{M}^{\gamma,\mathcal{O}}_{\beta,s+\alpha,\s}$.
There is $C(s)>0$ such that
\begin{align}
&\bral A M\brar_{\alpha+\beta,s,\s}^{\gamma,\mathcal{O}}
\leq C(s)\bral A\brar^{\gamma,\mathcal{O}}_{\alpha,s+|\beta|,\s}
\bral M\brar^{\gamma,\mathcal{O}}_{\beta,s+|\alpha|,\s}\,,\label{stim2}\\
&\bral({\rm Id}- \Pi_{N})M\brar^{\mu,\mathcal{O}}_{\beta,s,\s'}\leq 
\frac{C(s)e^{-(\s-\s')N}}{(\s-\s')^{d}}
\bral M\brar^{\gamma,\mathcal{O}}_{\beta,s,\s}\,
\quad 0<\s'<\s\,.\label{stim3}
\end{align}
Moreover, if $\alpha\leq\beta< 0$ then
\begin{equation}\label{stim1}
\bral AM \brar^{\gamma,\mathcal{O}}_{\beta,s,\s}
\leq C(s)\bral A\brar^{\gamma,\mathcal{O}}_{\alpha,s,\s}
\bral M\brar^{\gamma,\mathcal{O}}_{\beta,s,\s}\,.
\end{equation}
\end{lemma}

\begin{proof}
To prove \eqref{stim2} we need to bound the decay norms of
the operators $\mathcal{D}^{-\alpha-\beta}AM$ and 
$AM\mathcal{D}^{-\alpha-\beta}$. We have that
\[
AM\mathcal{D}^{-\alpha-\beta}=A\mathcal{D}^{-\alpha}
\mathcal{D}^{\alpha}(M\mathcal{D}^{-\beta})\mathcal{D}^{-\alpha}.
\]
Hence, by item $(iv)$ in Lemma \ref{DecayAlg}, 
$
|AM\mathcal{D}^{-\alpha-\beta}|_{s,\s}
{\leq_s}
\bral A\brar_{\alpha,s,\s}\bral M\brar_{\beta,s+|\alpha|,\s}\,.
$
Reasoning similarly one obtains the \eqref{stim2} for the Lipschitz norm 
$\bral \cdot\brar_{\alpha+\beta,s,\s}^{\gamma,\mathcal{O}}$ .
The \eqref{stim1} and \eqref{stim3} follow by Lemma \ref{DecayAlg}.
\end{proof}

 \begin{lemma}\label{AzioneDec}
Let $\beta\in \mathbb{R}$ and 
consider $A\in \mathcal{M}^{\gamma,\mathcal{O}}_{-\beta,s,\s}$.
Then
\begin{equation}\label{mito}
\|\mathcal{D}^{\beta} A h\|_{s,\s}\leq 
\bral A\brar_{-\beta,s,\s}
\|h\|_{s,\s}\,,\qquad \forall\; h\in \ell_{s,\s}\,.
\end{equation}
In particular (recall \eqref{spseq})
$A(\vphi) : h^{s}\mapsto h^{s+\beta}\,,\;\;\forall 
\vphi\in \mathbb{T}_{\s'}^{d}\,,\s'<\s\,, $ and
\begin{equation}\label{mito2}
\|A(\vphi)v\|_{s+\beta}\leq_s \frac{1}{(\s-\s')^{s_0+d}} \bral A\brar_{-\beta,s,\s}\|v\|_{s}\,
\end{equation}
for any $ v\in h^{s}$.
\end{lemma}
\begin{proof}
The \eqref{mito} follows by Lemma \ref{DecayAlg} and \eqref{decayNorm3}.
To prove \eqref{mito2} we reason as follows. We have
\begin{align}
\|A(\vphi)v\|^{2}_{s+\beta}&\stackrel{\eqref{spseq}}{=}
\sum_{k\in \mathbb{N}} \langle k\rangle^{2(s+\beta)}\|(A(\vphi)v)_{[k]}\|^{2} \leq_s
\sum_{k\in \mathbb{N}} \langle k\rangle^{2(s+\beta)}\Big(
\sum_{j\in\mathbb{N}}\|A_{[k]}^{[j]}(\vphi)\|_{\mathcal{L}(L^2)}\|v_{[j]}\|
\Big)^{2}\nonumber\\
&\leq_s\sum_{k,j\in \mathbb{N}} \langle k-j\rangle^{2s}
\|\langle k\rangle^{\beta}A_{[k]}^{[j]}(\vphi)\|^{2}_{\mathcal{L}(L^2)} \langle j\rangle^{2s}\|v_{[j]}\|^2 C(s)\label{mito3}
\end{align}
where 
$
C(s):=\sum_{j\in\mathbb{N}}\frac{ \langle k\rangle^{2s}}{ \langle k-j\rangle^{2s} \langle j\rangle^{2s}}\,.
$
It is easy to check that $C(s)<+\infty$. By Lemma \ref{decayspaziotempo}
and \eqref{mito3}
we have
\[
\|A(\vphi)v\|^{2}_{s+\beta}\leq_{s} |\mathcal{D}^{\beta}A(\vphi)|_{s}\|v\|_{s}
\leq_{s}
\frac{1}{(\s-\s')^{s_0+d}}\bral A\brar_{-\beta,s,\s}\|v\|_{s}
\]
which implies the thesis.
\end{proof}

\begin{lemma}\label{well-well}
Let $\beta<0$, consider 
$A\in \mathcal{M}_{\beta,s,\s}^{\gamma,\mathcal{O}}$ and assume
\begin{equation}\label{mito30}
C(s) \bral A\brar_{\beta,s,\s}^{\gamma,\mathcal{O}}\leq 1/2
\end{equation}
for some large $C(s)>0$.
Then the map $\Phi:={\rm Id}+\Psi$ defined in \eqref{exp}
satisfies
\begin{equation}\label{brackettech}
\bral \Psi\brar_{\beta,s,\s}^{\gamma,\mathcal{O}}
\leq_{s} \bral A\brar_{\beta,s,\s}^{\gamma,\mathcal{O}}\,,
\end{equation}
\end{lemma}
\begin{proof}
By \eqref{stim1}
we have
\[
\bral \Psi\brar_{\beta,s,\s}^{\gamma,\delta,\mathcal{O}}
\leq
\bral A\brar_{\beta,s,\s}^{\gamma,\mathcal{O}}
\sum_{p\geq 1}\frac{C(s)^{p}}{p!}
(\bral A\brar_{\beta,s,\s}^{\gamma,\mathcal{O}})^{p-1}\,,
\]
for some (large) $C(s)>0$. By the smallness condition \eqref{mito30}
one deduces   the bounds \eqref{brackettech}.
\end{proof}

We end with two results on the eigenvalues of Hermitian matrix.
\begin{lemma}\label{herm}
Let $\o\mapsto A(\o)$ be a Lipschitz mapping from $\mathcal O$ 
a compact set of $\R^d$ into the set  
Hermitian matrix of finite dimension $p$. 
Then the eigenvalues of $A(\o)$ can be 
ordered $\mu_1(\o)\leq \mu_2(\o)\leq\cdots\leq\mu_p(\o)$ 
in such a way each eigenvalue 
$\mu_j$ is  Lipschitz and
\begin{equation*}
|\mu_j|^{sup,\mathcal O}\leq \| A\|^{sup,\mathcal O}\,,\qquad
|\mu_j|^{lip,\mathcal O}\leq \| A\|^{lip,\mathcal O}\,, \qquad j=1,\cdots,p
\end{equation*}
where $\|\cdot\|$ denotes the operator norm.
\end{lemma}

\begin{proof}
This is a consequence of the Courant Fischer formula:
\[
\mu_j(A)=\min_{\rm{dim}V=k}\ \max_{\substack{x\in V\\\|x\|=1}}\langle Ax,x\rangle\, .
\]
\end{proof}
As a consequence we get the following.
\begin{corollary}\label{coromu}
 If $Z\in  \mathcal{M}^{\gamma,\mathcal{O}}_{-\beta,s,\s}$ is block diagonal then the eigenvalues of the block
$\mathcal{Z}_{[k]}^{[k]}$, denoted $\mu_{k,j}$, 
 $j=1,\ldots,d_{k}$, are  Lipschitz functions from $ \mathcal{O}$ into $\R$,
and satisfy
\begin{equation}\label{mu}
\sup_{\substack{k\in\mathbb{N}\\j=1,\cdots,d_k}}
\langle k\rangle^{\beta} 
|\mu_{k,j}|^{\gamma,\mathcal{O}}\leq
\bral\mathcal{Z}\brar_{-\beta,s,\s}^{\gamma,\mathcal{O}}\,.
\end{equation}
\end{corollary}

\end{document}